\documentclass[a4paper,oneside,reqno,12pt]{amsart}
\usepackage[latin1]{inputenc}
\usepackage[english]{babel}
\usepackage[margin=3.25cm]{geometry}
\usepackage{enumerate}
\usepackage{amsmath}
\usepackage{amsthm}
\usepackage{graphicx}
\usepackage{amsmath}
\usepackage{amssymb}
\usepackage{color}
\usepackage{enumerate}
\usepackage{graphics}
\usepackage{hyperref}

\newtheorem{remark}{Remark}[section]
\newtheorem{theorem}{Theorem}[section]
\newtheorem{lemma}{Lemma}[section]
\newtheorem{corollary}{Corollary}[section]

\DeclareMathOperator*{\argmin}{arg\,min}
\newcommand{\ve}{\varepsilon}
\newcommand{\mop}[1]{{\mathop{\mathrm{#1}}}}
\newcommand{\inv}{^{-1}}

\title[Homogenization of a locally periodic oscillating boundary]{Homogenization of a locally periodic oscillating boundary}
\author{S. Aiyappan}
\address{Fraunhofer ITWM, Kaiserslautern, Germany}
\email{aiyappan@itwm.fraunhofer.de, aiyappan.iisc@gmail.com}
\author{K. Pettersson}
\address{Chalmers University of Technology, Sweden}
\email{klaspe@chalmers.se}
\subjclass[2000]{80M35,80M40,35B27,49J20}
\date{\today}
\keywords{Homogenization, asymptotic analysis, periodic unfolding, locally periodic boundary, oscillating boundary.}
\begin{document}

\begin{abstract}
This paper deals with the homogenization of a mixed boundary value problem for the Laplace operator in a domain
with locally periodic oscillating boundary.
The Neumann condition is prescribed on the oscillating part of the boundary, and the Dirichlet condition on a separate part.
It is shown that the homogenization result holds in the sense of weak $L^2$ convergence of the solutions and their flows,
under natural hypothesis on the regularity of the domain.
The strong $L^2$ convergence of average preserving extensions of the solutions and their flows is also considered.
\end{abstract}

\maketitle

\section{Introduction}\label{Sec-Intro} 
This paper is concerned with the homogenization of a boundary value problem for the Laplace operator on
a domain in $\mathbb{R}^2$ with locally periodic oscillating boundary.
Specifically, the domain is given by
\begin{align*}
\Omega^\varepsilon=\big\{x \in \mathbb{R}^2: 0<x_1<1, \,\, 0<x_2<\eta\big(x_1, \frac{x_1}{\varepsilon}\big)\big\},
\end{align*}
where $\eta$ is a positive Lipschitz continuous function which is periodic in the second variable, and $\ve$ is a small positive parameter (see Figure~\ref{fig:domain}(a)).
Certain further requirements are imposed on $\eta$ to ensure a particularly simple structure of the homogenized problem.
On the oscillating part of the boundary the Neumann condition is prescribed, and on a separate part the Dirichlet condition,
and the data are assumed to be $L^2$.
We are interested in the asymptotic behavior of the solutions $u^\ve$ to the boundary value problem, and their flows, as $\ve$ tends to zero.

Domains with oscillating boundaries have attracted particular interest in the case where the domain is thin,
that is to say when homogenization and dimension reduction may take place.
This is natural in the mathematical physics program of the derivation of lower-dimensional theories
from three-dimensional (see e.g.~\cite{ball2002some}).
There is a rich literature on thin heterogeneous domains (see \cite{nazarov2001asymptotic,oleinik2009mathematical,ArVi-SIMA-16} and the references therein).
Asymptotic analysis in thin domains with locally periodic oscillating boundary was conducted 
in for example~\cite{chechkin1999boundary,mel2010asymptotic,borisov2010asymptotics,arrieta2011homogenization,akimova2004asymptotics,ArVi-SIMA-16,pettersson2017two,ArVi-JMAA-17}.

There are many works on homogenization in periodically oscillating domains with pillar type oscillations of fixed
amplitude where the cross-section of each pillar is constant in the vertical direction.
For the literature on pillar type oscillations, we refer to \cite{esposito1997homogenization,NanRavBid1,GaMe-JDE-18} and the references therein.
Oscillating boundary domains with non-uniform cylindrical pillars,
that is when the cross-sections of the pillars are varying in the vertical direction,
have been considered in \cite{Ga-RdM-94,Me-MMA-08,AiNaRa-CV,MaNaRa2018}.
In the mentioned works, the top boundary of the pillars have been assumed to be flat, that is the measure of the cross-section of
each pillar at the maximum height is assumed to be positive, and also the base of each pillar assumed to be flat.
There are few works on non-flat top boundaries, and \cite{BrCh-RdM-97,AiNaRa-CCM} stand out. 
In \cite{BrCh-RdM-97}, the authors restrict the boundary graph functions to be smooth, periodic, and to have a unique maximum in each period.
In \cite{GaGuMu-ARMA}, the authors consider an oscillating domain without explicit periodicity assumption and the base
of each uniform pillar is allowed to be non-flat.
Locally periodic flat pillar type domains were considered in \cite{DaPe-DCDS}, with respect to width and height.
The works of Mel'nyk and his collaborators (c.f. \cite{mel1997asymptotics,mel1999homogenization,de2005asymptotic,DurMel,mel2015asymptotic,gaudiello2019homogenization})
appear to have had a strong influence on later developments, after the initial works of Brizzi and Chalot~\cite{BriCha78,BrCh-RdM-97}.

In this paper, the oscillating domain $\Omega^\ve$ is a bounded region partially bounded by the graph of a locally periodic Lipschitz function $\eta(x_1,x_1/\ve)$, where $\eta : [0,1] \times \mathbb{T} \to \mathbb{R}$,
$\mathbb{T}$ is the 1-torus,
and $\ve$ is a small parameter.
These assumptions on $\eta$ ensure that the domain is connected and Lipschitz. 
In particular, the domain is not thin in the direction normal to oscillation, and not of pillar-type, and 
no assumptions are made on the flatness.
Under these assumptions, the domain naturally becomes asymptotically disconnected (in the $x_1$ direction)
between two curves, one that appears 
as a part of the limiting boundary and one that appears as an interior interface, as $\ve$ tends to zero.
The assumptions on $\eta$ guarantee that these curves are graphs of Lipschitz functions.

The analysis simplifies considerably with the Brizzi-Chalot condition of a single bump in each period.
It appears to be worthwhile to remark that it is not generally true that as soon as there is more than
one bump in each period, even if one restricts to the smallest possible periodicity cell,
the asymptotic limit will not be decoupled (fast and slow variables).
An example that shows that the connectedness of the sections is not necessary is given.

The expected influence of the domain oscillations on  
the asymptotic behavior of the solutions 
is that since the Laplace operator is local and the periodicity of the domain makes a region
asymptotically become non-connected in the $x_1$ direction,
$\partial/\partial x_1$ cannot be present in the homogenized equation in that region.
This will show in the homogenized boundary value problem.

The result of this paper is the homogenization of the domain for the particular boundary value problem
in which the heterogeneity is only in the domain.
Three properties of homogenization are shown.
Namely, (i) the weak $L^2$ convergence of the zero-extended solutions and their flows (Theorem \ref{tm:homogenization}),
(ii) that the error of the zeroth approximation of the solutions and their flows converge strongly to zero in $L^2$ restricted to the 
oscillating domain (Theorem \ref{tm:justification}), and (iii) the strong $L^2$ convergence of average preserving extensions of the solutions and their flows (Theorem \ref{tm:homogenization2}).
In regard to (iii), in~\cite{BrCh-RdM-97} reflection extensions were constructed and used, while the extension we use is the one used in~\cite{lipton1990darcy,DaPe-DCDS}.\\

The analysis methods we use are standard techniques of asymptotic analysis and homogenization in particular.
The method of homogenization is outlined in \cite{jikov1994sm}.
The method of periodic unfolding is described in \cite{cioranescu2002periodic,CiDaGr-SIMA-08,cioranescu2018periodic},
which is closely related to the notion of two-scale convergence \cite{nguetseng1989general,allaire1992homogenization,zhikov2004two},
a generalization of weak convergence.
Some works in which the unfolding method was used extensively in problems with oscillating boundary are
\cite{BlGaGr-JMPA-1,BlGaGr-JMPA-2,DaPe-DCDS,AiNaRa-CV}.
The homogenized problem is of degenerate elliptic type, as will be described below (c.f.~\cite{BrCh-RdM-97}).
The classical theory of Sobolev spaces for such is outlined in \cite{kufner1985weighted, kufner1987some}.
For the method of asymptotic expansions we refer to~\cite{bakhvalov1974averaged,bakhvalov1975averaging,bakhvalov1975averaging2,papanicolau1978asymptotic,oleinik1975,sanchez1980non,bakhvalov2012homogenisation,spagnolo2007sulla,berdichevsky,babuska,marchenko1964boundary}.

The rest of this paper is organized as follows.
The problem statement and the results are presented in Section~\ref{sec:problem}.
The homogenized problem is derived using formal asymptotic expansions of the solutions in Section~\ref{sec:expansion}.
In Section~\ref{sec:zigzag}, an example is provided that shows that the first term in the asymptotic expansion to $u^\ve$ may be globally regular even in a degenerating case.
In Section~\ref{sec:unfolding}, the mean value property for the periodic unfolding is presented, adjusted to the present problem. The homogenization result of weak convergence of the solutions and their flows is established in Section~\ref{sec:homogenization}.
In Section~\ref{sec:justification}, the convergence of energy is shown. 
In Section~\ref{sec:averagepreserving}, the strong convergence of the extended solutions and their flows
using average preserving extensions is shown.
A numerical example, illustrating the rate of convergence is presented in Section~\ref{sec:numerical}.
In Section~\ref{sec:nonconnected}, some information about a case of non-connected sections is provided.

\section{Problem statement and results}\label{sec:problem}

In this section we state the boundary value problem and present the results.

With strictly positive Lipschitz $\eta : [0,1] \times \mathbb{T} \to \mathbb{R}$,
the sequence of Lipschitz domains with periodically oscillating boundaries is for each $\ve = 1/k$, $k = 1, 2, \ldots$, defined by
\begin{align*}
\Omega^\varepsilon & = \big\{ x \in \mathbb{R}^2 : 0 < x_1 < 1 ,\,\, 0 < x_2 < \eta\big(x_1, \frac{x_1}{\ve}\big) \big\}.
\end{align*}
Here, $\mathbb{T}$ denotes the one-dimensional torus realized as $(0,1)$.
See Figure~\ref{fig:domain}(a) for an illustration of $\Omega^\ve$.
Additional restrictions on $\eta$ will be added below in order to ensure homogenization.

Let $u^\varepsilon \in H^1(\Omega^\varepsilon, \Gamma)$ be the sequence of solutions to the following mixed boundary value problem:
\begin{align}\label{eq:originalproblem}
-\Delta u & = f \quad \text{ in } \Omega^\varepsilon, \notag\\
u & = 0 \quad \text{ on } \Gamma, \\
\nabla u \cdot \nu & = 0 \quad \text{ on } \partial \Omega^\ve \setminus \Gamma,\notag
\end{align}
with $f$ in $L^2( \Omega )$, where $\Omega^\ve \subset \Omega$.
Here $\nu$ denotes the outward unit normal to the domain,
and $H^1(\Omega^\varepsilon, \Gamma)$ the functions in $H^1(\Omega^\varepsilon)$ with zero trace on $\Gamma = (0,1) \times \{0\}$.
Our goal is to describe the asymptotic behavior of the solutions $u^\ve$ to~\eqref{eq:originalproblem} as $\ve$ tends to zero.

\begin{figure}[!hb] 
    \centering
    \begin{minipage}{.5\textwidth}
        \centering
        \includegraphics[trim=40mm 10mm 30mm 5mm, clip,height=8cm]{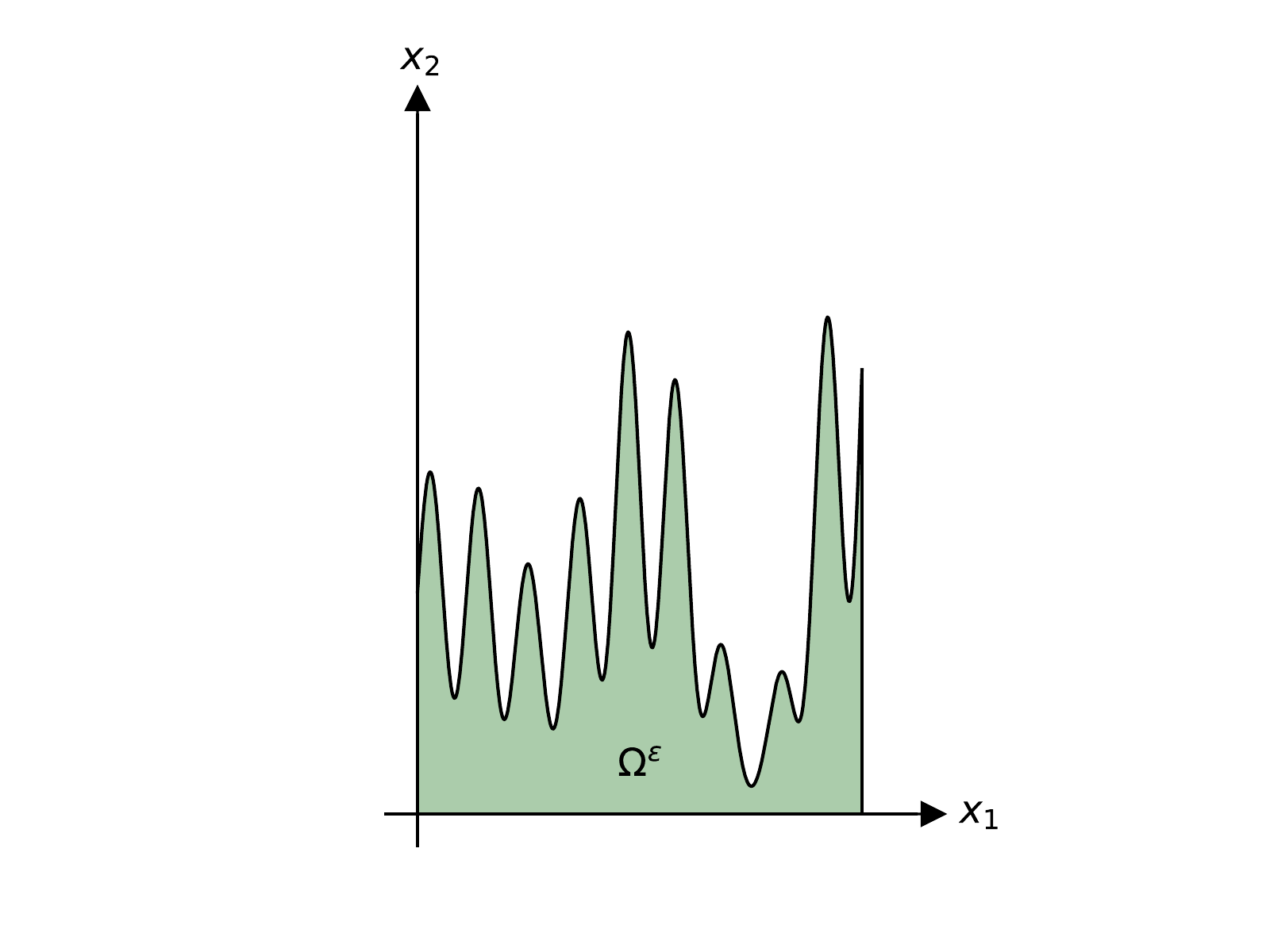}\\
        (a)
    \end{minipage}%
    \begin{minipage}{0.5\textwidth}
        \centering
        \includegraphics[trim=40mm 10mm 30mm 5mm, clip,height=8cm]{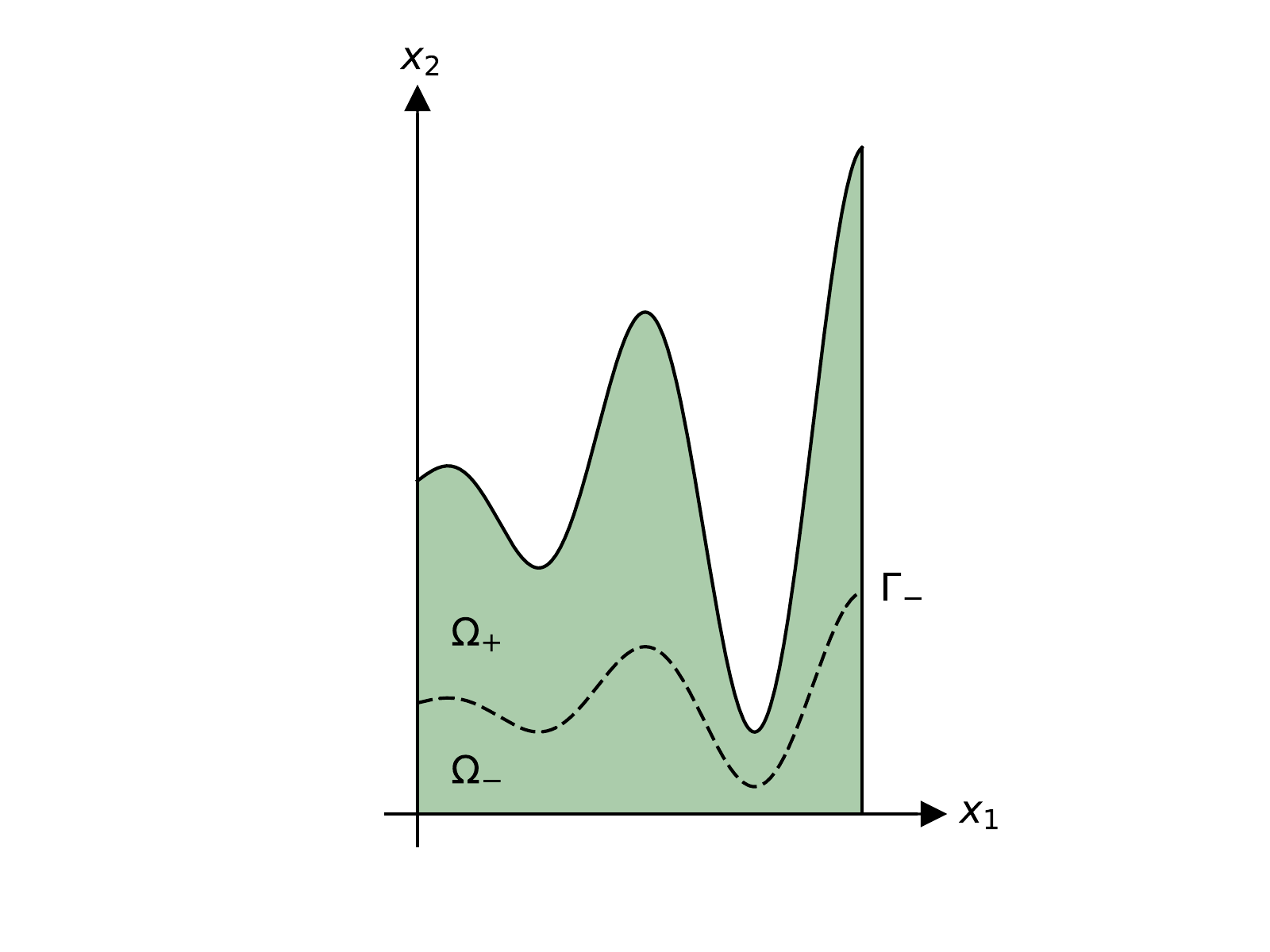}\\
        (b)
    \end{minipage}
    \caption{A locally periodic domain $\Omega^\ve$ (a), $\ve = 1/8$, and the corresponding homogeneous domain $\Omega$ (b),
                 with $\Gamma_-$ marked with a dashed line separating the regions $\Omega_+$ and $\Omega_-$.
                 The particular function $\eta$ is 
                 $\eta(x,y) = (1 + x \cos(4\pi x))(1 + \sin(2\pi(x+y))/2)$.
                 }
    \label{fig:domain}
\end{figure}

The solutions $u^\ve$ will be approximated in terms of the solution $u^0$ to the homogenized problem:
\begin{align}\label{eq:limitproblem}
-\mop{div}(h A^0 \nabla u) & = h f \quad \text{ in } \Omega, \notag\\
u & = 0 \,\,\,\quad \text{ on } \Gamma, \\
h A^0 \nabla u \cdot \nu & = 0 \,\,\,\quad \text{ on } \partial \Omega \setminus \Gamma, \notag
\end{align}
where the coefficients $A^0$ and the domain $\Omega$ are defined as follows.
In terms of the Lipschitz functions
\begin{align*}
\eta_-(x) & = \min_y \eta(x,y), &
\eta_+(x) & = \max_y \eta(x,y),
\end{align*}
the domain
\begin{align*}
\Omega & = \{ x \in \mathbb{R}^2 : 0 < x_1 < 1 , \,\, 0 < x_2 < \eta_+(x_1) \}
\end{align*}
is separated into the regions
\begin{align*}
\Omega_- & = \{ x \in \mathbb{R}^2 : 0 < x_1 < 1 , \,\, 0 <  x_2 < \eta_-(x_1) \}, \\
\Omega_+ & = \{ x \in \mathbb{R}^2 : 0 < x_1 < 1 , \,\,   \eta_-(x_1) <  x_2 < \eta_+(x_1) \},
\end{align*}
with interior interface $\Gamma_- = \partial \Omega_- \cap \partial \Omega_+$. 
The effective matrix is 
\begin{align}\label{eq:effmatrix}
A^0 & = \left(
\begin{matrix}
\chi_{\Omega_-}& 0 \\ 0 & 1
\end{matrix}\right).
\end{align}

An illustration of $\Omega$, with the regions $\Omega_+$ and $\Omega_-$, and the interface $\Gamma_-$, indicated
is shown in Figure~\ref{fig:domain}(b),
corresponding to the domain $\Omega^\ve$ in Figure~\ref{fig:domain}(a).

Let $h$ denote what we call the density of $\Omega^\ve$ in $\Omega$:
\begin{align}
Y(x) & = \{ y : x_2 < \eta(x_1, y) \},\label{eq:Ydef}\\
h(x) & = |Y(x)|.\label{eq:hdef}
\end{align}
In $\Omega_-$ the density is $h = 1$. A part of the upper boundary, denoted by  $\Gamma_a$ is given by
\begin{align*}
\Gamma_a = \left \{ (x_1,x_2)~| ~ x_1 \in (0,1),~x_2= \eta_+(x_1),~h(x_1,x_2)=0 \right\}.
\end{align*}

In order to ensure that homogenization takes place, the following hypotheses will be used:
\begin{enumerate}[(H1)]
\item $Y(x)$ is connected, $x \in \Omega$.
\item $Y(x) = k/N + Y(x)$, $x \in \Omega$, $k \in \mathbb{Z}$, for some natural number $N \ge 1$, 
and $Y(x) = y_0 - Y(x)$ for some $y_0 \in [0,1/N)$, 
and $Y_0(x) = \{ y \in Y(x) : y_0 \le y \le y_0 + 1/(2N) \}$ is connected, $x \in \Omega$.
\end{enumerate}

The hypothesis (H1) means that there is only one so-called pillar or bump in each period,
and (H2) means restricting to a fundamental symmetry cell with respect to some translations and 
mirror symmetry.
The hypothesis (H1) is stronger than (H2), and it gives a sufficient condition for homogenization.
The hypothesis (H2) is included here to illustrate that it is not necessary there is only one 'bump' in each period even
if one uses the smallest periodicity cell.
The graph of a function $\eta$ that satisfies (H2) but not (H1) is illustrated in Figure~\ref{fig:domainnumerical}.

The assumption that $\eta$ is strictly positive ensures that the segment $\Gamma = (0,1) \times \{0\}$
is separated from the graph of $\eta(x_1,x_1/\ve)$, so $\Omega_-$ is a nonempty connected Lipschitz domain.
The subdomains $\Omega_+$ and $\Omega_-$ have been chosen such that $\Omega_+$ covers
the periodic region of $\Omega^\ve$, 
and $\Omega_+$ is of positive measure if $\eta(x,y)$ is non-constant in $y$ for at least one $x$.

Denote by $L^2(\Omega, h)$ the Lebesgue space $\{ v : \int_{\Omega} v^2 h \,dx < \infty \}$, and $W(\Omega, \Gamma)$
the Sobolev space
\begin{align}\label{eq:WGamma}
W(\Omega, \Gamma)
     & = 
     \big\{  v \in L^2(\Omega, h) :
     A^0 \nabla v \in L^2(\Omega, h)
, \,
     v = 0 \text{ on } \Gamma \big\},
\end{align}
where $A^0$ is defined in~\eqref{eq:effmatrix} and $h$ is defined in~\eqref{eq:Ydef}-\eqref{eq:hdef}.
The homogenized problem \eqref{eq:limitproblem} has a unique solution $u^0 \in W(\Omega, \Gamma)$.

We will first establish the homogenization of \eqref{eq:originalproblem},
that is the convergence of the solutions $u^\ve$ and their flows $\nabla u^\ve$
to the solution $u^0$ to the homogenized problem \eqref{eq:limitproblem} and its flow $hA^0 \nabla u^0$,
under hypothesis that (H1) holds.

Let $u^\ve \in H^1(\Omega^\ve, \Gamma)$ be the solutions to \eqref{eq:originalproblem},
and let $u^0 \in W(\Omega, \Gamma)$ be the solution to \eqref{eq:limitproblem}.
In Theorem~\ref{tm:homogenization}, it is shown under hypothesis (H1) that
\begin{align*}
&\widetilde{u^\ve} \rightharpoonup h u^0 \quad \text{ weakly in } L^2(\Omega),\\
&\widetilde{\nabla u^\ve} \rightharpoonup h A^0 \nabla u^0 \quad \text{ weakly in } L^2(\Omega),
\end{align*}
as $\ve$ tends to zero.
Here tilde denotes extension by zero.

At this point we know that there will be oscillations in $\widetilde{u^\ve}$ in the upper part $\Omega_+$ due to the periodicity of $\Omega_\ve$,
while no oscillations in the lower part $\Omega_-$ due to the compact embedding of $H^1(\Omega_-)$ into $L^2(\Omega_-)$.
Moreover, we cannot at this point exclude the possibility of oscillations in the solutions $\widetilde{u^\ve}$ due to
something else, as the above weak convergences may be expressed as weak unfolding or two-scale convergence. 
A next step is to check whether the strong unfolding convergence 
holds 
for the solutions $u^\ve$ and their flows.

The weak convergence of the zero extensions in the upper part $\Omega_+$ cannot be strong, unless the limit is zero.
To this end we first describe the error measured in the oscillating domain.
Not only are the unfoldings of the solutions $u^\ve$ and their flows strongly converging,
there are no oscillations.

For $u^\ve$ the solutions to \eqref{eq:originalproblem} and $u^0$ the solution to \eqref{eq:limitproblem},
in Theorem~\ref{tm:justification} it is shown under hypothesis (H1) that
\begin{align*}
& \| u^\ve - u^0 \|_{ L^2(\Omega^\ve, \,h) } \to 0,\\
& \| \nabla u^\ve - A^0 \nabla u^0 \|_{ L^2(\Omega^\ve, \,h) } \to 0,
\end{align*}
as $\ve$ tends to zero.

Turning back to the question of homogenization, there is an extension of the solutions $u^\ve$ and their flows
for which strong convergence holds in $L^2(\Omega)$.

For $u^\ve$ the solutions to \eqref{eq:originalproblem} and $u^0$ the solution to \eqref{eq:limitproblem},
in Theorem~\ref{tm:homogenization2} it is shown that when the functions are extended in a way preserving their 
average, under hypothesis (H1) (c.f. \cite{lipton1990darcy,DaPe-DCDS}),
\begin{align*}
&\widetilde{u^\ve}^m \to u^0 \quad \text{ strongly in } L^2(\Omega, h),\\
&\widetilde{\nabla u^\ve}^m  \to A^0 \nabla u^0 \quad \text{ strongly in } L^2(\Omega, h ),
\end{align*}
as $\ve$ tends to zero, when $\sim{\!m}$ denotes the particular extension (\eqref{eq:avgext1}, \eqref{eq:avgext2} in Section~\ref{sec:averagepreserving}).
It is remarked that the above mentioned convergences also hold under hypothesis (H2).

Without the hypotheses (H1), (H2), under a slightly milder restriction on the domain,
 in Theorem~\ref{tm:weaknonconnected} it is shown that
the solutions $u^\ve$ to~\eqref{eq:originalproblem} converge to the solution $u^0$ to the limit problem~\eqref{eq:limitproblemstrongdisconnected}, with $A^0$ given by~\eqref{eq:A0nonconnected}, in the
sense that
\begin{align*}
\widetilde{u^\ve} & \rightharpoonup \int_{Y(x)} u^0 \,dy \quad \text{ weakly in } L^2(\Omega), \\
\widetilde{\nabla u^\ve} & \rightharpoonup \int_{Y(x)} A^0 \nabla_x u^0 \,dy \quad \text{ weakly in } L^2(\Omega),
\end{align*}
as $\ve$ tends to zero.
One also has strong unfolding convergence.
We do not analyze further the convergence in the case of non-connected sections.

In the case of the homogeneous Dirichlet condition on the oscillating part of the boundary,
the limit is trivial in the oscillating part.
This case was studied in~\cite{amirat2004asymptotic}.

\begin{remark}[The effect of anisotropy and oscillating coefficients]\label{rem:osccoef}
To illustrate the effect of anisotropy and oscillating coefficients on the asymptotic behavior of the solutions $u^\ve$ to \eqref{eq:originalproblem},
one can consider the following elliptic model problem:
\begin{align*}
-\mop{div}\big( A\big(x, \frac{x_1}{\ve}\big) \nabla u \big) & = f \quad \text{ in } \Omega^\varepsilon, \notag\\
u & = 0 \quad \text{ on } \Gamma, \\
A\big(x, \frac{x_1}{\ve}\big) \nabla u \cdot \nu & = 0 \quad \text{ on } \partial \Omega^\ve \setminus \Gamma,\notag
\end{align*}
under the assumption that the not necessarily symmetric matrix $A \in C^1(\overline{\Omega}, L^\infty(\mathbb{T}))$
satisfies the ellipticity condition
\begin{align*}
A(x,y) \xi \cdot \xi \ge C |\xi|^2, \quad x \in \Omega, \, y \in \mathbb{T}, \, \xi \in \mathbb{R}^2.
\end{align*}
Suppose further that hypothesis (H2) is satisfied, and that the coefficient matrix satisfies the corresponding symmetry conditions:
$A(x,y) = A(x,k/N+y)$, $k \in \mathbb{Z}$, and $A(x,y) = A(x,y_0-y)$, for $x \in \Omega_+$, where $N$ and $y_0$ are the 
same as for $Y(x)$ in (H2).
Then the weak limits of $\widetilde{u^\ve}$ and $\widetilde{\nabla u^\ve}$ are $h u^0$ and $A^0 u^0$, respectively, in $L^2(\Omega)$, as $\ve$ tends to zero, where $u^0$ solves the limit problem
\begin{align*}
-\mop{div}(A^0 \nabla u) & = h f \quad \text{ in } \Omega, \notag\\
u & = 0 \,\,\,\quad \text{ on } \Gamma, \notag \\
A^0 \nabla u \cdot \nu & = 0 \,\,\,\quad \text{ on } \partial \Omega \setminus \Gamma,
\end{align*}
where the entries $A^0_{ij}$ of the effective matrix $A^0$ are given by, here including $h$,
\begin{align*}
A^0_{11} & = \chi_{\Omega_-} \Big( \int_{Y(x)} \frac{1}{A_{11}} \,dy \Big)^{-1}, \\
A^0_{12} & = A^0_{11} \int_{Y(x)} \frac{A_{12}}{A_{11}} \,dy, \\
A^0_{21} & = A^0_{11} \int_{Y(x)} \frac{A_{21}}{A_{11}} \,dy, \\
A^0_{22} & = A^0_{11} \int_{Y(x)} \frac{A_{12}}{A_{11}} \,dy \int_{Y(x)} \frac{A_{21}}{A_{11}} \,dy + \int_{Y(x)} \frac{\mop{det }A}{A_{11}} \,dy,
\end{align*}
where $A_{ij}$ denote the entries of the local matrix $A$.
In the fixed region $\Omega_-$, $A^0$ is the classical effective matrix for layered materials because there $Y(x) = \mathbb{T}$, while in the oscillating region $\Omega_+$ all but the last term in $A_{22}^0$ vanish due to the insulation in the $x_1$ direction.
In particular, for $-\Delta$ the effective matrix reduces to the matrix given in~\eqref{eq:effmatrix}.
A similar statement can be made for the case of non-connected sections analogous to Theorem~\ref{tm:weaknonconnected}.
\end{remark}

\section{Asymptotic expansions for the solutions}\label{sec:expansion}

To derive the homogenized equation \eqref{eq:limitproblem} for the solutions $u^\ve$ to the equation
\begin{align}\label{eq:originalequation}
-\Delta u^\ve & = f  \quad \text{ in } \Omega^\ve,
\end{align}
asymptotic expansions may be used, in the form of a formal power series in the small parameter $\ve$.
In the problem \eqref{eq:originalequation}, in the periodic region $\Omega^\ve_+$ of the domain $\Omega$,
$\ve$ represents both the periodicity in the $x_1$ direction,
as well as the order of magnitude of the widths of the pillars with homogeneous Neumann condition on their sides.
In the fixed region $\Omega_-$ of the domain, $\ve$ represents the periodicity on the interface $\Gamma_-$.

The above reasoning leads us to consider the following Bakhvalov ansatz for inner expansion in the periodic region $\Omega^\ve_+$:
\begin{align}\label{eq:series}
u^\ve(x) \sim \big( u^0_+(x_2,y) + \ve u^1_+(x_2,y) + \ve^2 u^2_+(x_2,y) \big)\Big|_{\displaystyle y = \frac{x_1}{\ve}},
\end{align}
where $u^i_+$ are assumed to be periodic in $y$.

Let $\chi(x,y)$ denote the characteristic function of the set
\begin{align*}
\{ (x,y) : 0 < x_1 < 1, \,\, 0 < x_2 < \eta(x_1, y)  \}.
\end{align*}
Then the characteristic function of $\Omega^\ve$ is
\begin{align*}
\chi_{\Omega^\ve}(x) = \chi\big(x, \frac{x_1}{\ve}\big).
\end{align*}
We write \eqref{eq:originalequation} in the homogeneous domain $\Omega$ as follows:
\begin{align}\label{eq:originalequation2}
-\sum_{i} \frac{\partial}{\partial x_i} \big( \chi\big(x,\frac{x_1}{\ve}\big) \frac{\partial}{\partial x_i}u^\ve \big)
& =
\chi\big(x,\frac{x_1}{\ve}\big) f  \quad \text{ in } \Omega.
\end{align}

With the ansatz \eqref{eq:series},
\begin{align*}
\Delta \sim \frac{1}{\ve^2}\frac{\partial^2}{\partial y^2} + \frac{\partial^2}{\partial x_2^2}.
\end{align*}
A substitution of \eqref{eq:series} into \eqref{eq:originalequation2} 
and collecting similar powers of $\ve$ result in the following equations for the initial powers of $\ve$:
\begin{align*}
\ve^{-2}: \quad &  -\frac{\partial}{\partial y} \big(\chi \frac{\partial u^0_+}{\partial y} \big) = 0,     \\
\ve^{-1}: \quad &   -\frac{\partial}{\partial y} \big(\chi \frac{\partial u^1_+}{\partial y} \big) = 0,     \\
\ve^{0}: \quad &    -\frac{\partial}{\partial y} \big(\chi \frac{\partial u^2_+}{\partial y} \big)
= \frac{\partial}{\partial x_2} \big(\chi \frac{\partial u^0_+}{\partial x_2} \big) + \chi f    .
\end{align*}
The equation for power $\ve^{-2}$ suggests that $u^0_+$ is independent of $y$, under the assumption that $Y(x)$ is connected, or some discrete symmetry in the problem such as (H2).
Viewing the above equations as definitions of $u^i_+$ in $\mathbb{T}$ with $x$ as a parameter,
the compatibility condition for $u^2_+$ may be read off from the $\ve^0$ equation:
\begin{align*}
\int_\mathbb{T} \big( \frac{\partial}{\partial x_2} \big(\chi \frac{\partial u^0_+}{\partial x_2} \big) + \chi f     \big) \,dy & = 0.
\end{align*}
Because
\begin{align*}
\int_\mathbb{T} \chi(x,y) \,dy & = \int_{Y(x)} dy = |Y(x)| = h(x),
\end{align*}
where
\begin{align*}
Y(x) & = \{ y : x_2 < \eta(x_1, y) \},\\
h(x) & = |Y(x)|,
\end{align*}
the compatibility condition is
\begin{align}\label{eq:limitequation}
- \frac{\partial}{\partial x_2} \big( h \frac{\partial u^0_+}{\partial x_2} \big) & = hf \quad \text{ in } \Omega_+.
\end{align}
The equation \eqref{eq:limitequation} is the homogenized equation for $u^\ve$ in the periodic part $\Omega_+$ of $\Omega$.

In the fixed region $\Omega_-$ of the domain, the equation \eqref{eq:originalequation} reads
\begin{align}\label{eq:originalequation3}
-\Delta u^\ve & = f \quad \text{ in } \Omega_-.
\end{align}
Consider the following ansatz for inner expansion in the fixed region $\Omega_-$:
\begin{align}\label{eq:series2}
u^\ve(x) \sim \big( u^0_-(x,y) + \ve u^1_-(x,y) + \ve^2 u^2_-(x,y) \big)\Big|_{\displaystyle y = \frac{x_1}{\ve}},
\end{align}
where $u^i_-$ are assumed to be periodic in $y$.
With the ansatz \eqref{eq:series2},
\begin{align*}
\Delta \sim \frac{1}{\ve^2}\frac{\partial^2}{\partial y^2} + \frac{2}{\ve}\frac{\partial^2}{\partial y \partial x_1} + \Delta_x.
\end{align*}
A substitution of \eqref{eq:series2} into \eqref{eq:originalequation3} 
and collecting similar powers of $\ve$ result in the following equations for the initial powers of $\ve$:
\begin{align*}
\ve^{-2}: \quad &  -\frac{\partial^2 u^0_-}{\partial y^2} = 0,     \\
\ve^{-1}: \quad &  -\frac{\partial^2 u^1_-}{\partial y^2} = 2\frac{\partial^2 u^0_-}{\partial y \partial x_1},     \\
\ve^{0}: \quad &  -\frac{\partial^2 u^2_-}{\partial y^2} = 2\frac{\partial^2 u^1_-}{\partial y \partial x_1} + \Delta_x u^0_- + f.
\end{align*}
The equation for power $\ve^{-2}$ suggests that $u^0_-$ is independent of $y$.
By the periodicity of $u^1_-$, the compatibility condition for $u^2_-$ reads
\begin{align*}
-\Delta u^0_- & = f \quad \text{ in } \Omega_-.
\end{align*}

In order $u^0 = \chi_{\Omega_-} u^0_- + \chi_{\Omega_+} u^0_+$ to be globally defined, an interface condition is needed on
$\Gamma_- = \partial \Omega_- \cap \partial \Omega_+$.
If one requires continuity of $u^0$ and its flow on $\Gamma_-$, one must have $u^0_+ = u^0_-$ on $\Gamma_-$, as well as 
\begin{align*}
\Big(
\begin{matrix}
1 & 0 \\ 0 & h
\end{matrix}
\Big)
\nabla u^0_+ \cdot \nu - \nabla u^0_- \cdot \nu = 0,
\end{align*}
where $\nu$ is one of the unit normals on $\Gamma_-$, and the upper left entry of the matrix in the first term is arbitrary.
This condition becomes explicitly, 
\begin{align*}
h \frac{\partial u^0_+}{\partial x_2} \nu_2 - \nabla u^0_- \cdot \nu = 0,
\end{align*}
which may be expressed as 
\begin{align*}
[hA^0 \nabla u^0 \cdot \nu] = 0 \quad \text{ on } \Gamma_-,
\end{align*}
where $[\cdot]$ denotes the jump on $\Gamma_-$, and $A^0$ is given by~\eqref{eq:effmatrix}.

\section{Example of behavior of the solutions in a degenerating case}\label{sec:zigzag}

The solution $u^0$ to the homogenized problem~\eqref{eq:limitproblem} may belong to $H^1(\Omega)$
even if $h(x)$ tends to zero as $x$ approaches
$M = \{ x : h(x) = 0 \} \subset \partial \Omega$, as the following example illustrates.
Consider the case $f = 1$ with $\eta(x,y) = \hat \eta(y)$ the piecewise linear function 
given by $\hat \eta(0) = 1$ and $\hat \eta(1/2) = 2$, that is
$\hat \eta(y) = 2 - 2|y - 1/2|$.
The homogeneous domains are
$\Omega = (0,1) \times (0,2)$, 
$\Omega_+ = (0,1) \times (1,2)$,
$\Omega_- = (0,1) \times (0,1)$,
and the interface is $\Gamma_- = (0,1) \times \{ 1 \}$.
In Figure~\ref{fig:domainz}, illustrations of the domains $\Omega^\ve$ and $\Omega$ are shown.
In this case,
\begin{align*}
h(x) & = 
\begin{cases}
2 - x_2 & \text{ in } \Omega_+, \\
1 & \text{ in } \overline{\Omega_-},
\end{cases}
\end{align*}
and
\begin{align*}
h &= 0 \quad \text{ on } M = \{ x : h(x) = 0 \} = (0,1) \times \{ 2 \}, \\
h &= 1 \quad \text{ on } \Gamma_-.
\end{align*}
The solution to the homogenized problem \eqref{eq:limitproblem} is
\begin{align*}
u^0(x) & = 
\begin{cases}
1/4 + x_2 - x_2^2/4 & \text{ in } \Omega_+, \\
3x_2/2 - x_2^2/2 & \text{ in } \overline{\Omega_-}.
\end{cases}
\end{align*}
In particular, $u^0 \in H^1(\Omega)$ and it has continuous gradient over the interface $\Gamma_-$.

\begin{figure}[!hb] 
    \centering
    \begin{minipage}{.5\textwidth}
        \centering
        \includegraphics[trim=35mm 10mm 15mm 0mm, clip,height=6.25cm]{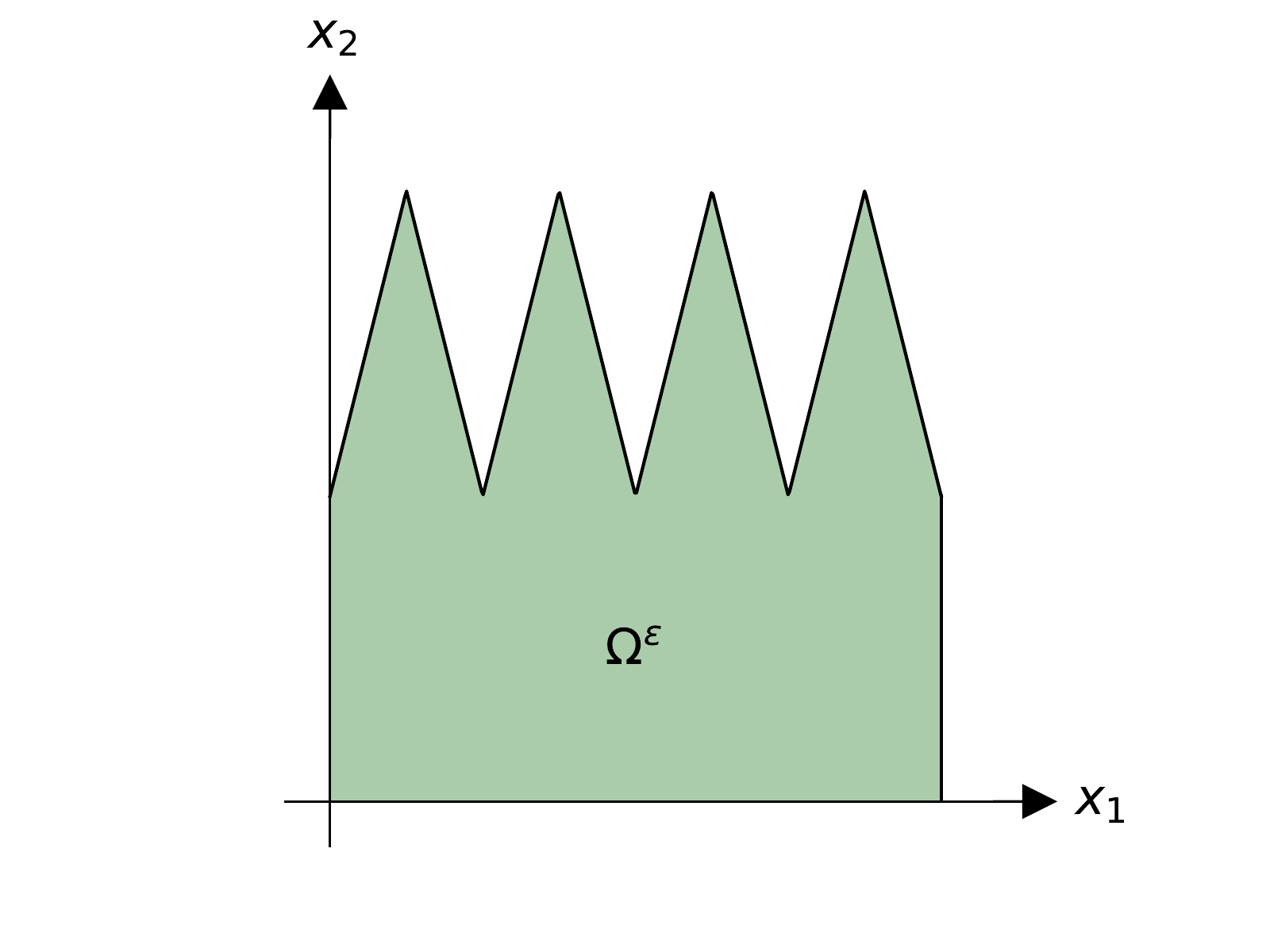}\\
        (a)
    \end{minipage}%
    \begin{minipage}{0.5\textwidth}
        \centering
        \includegraphics[trim=35mm 10mm 15mm 0mm, clip,height=6.25cm]{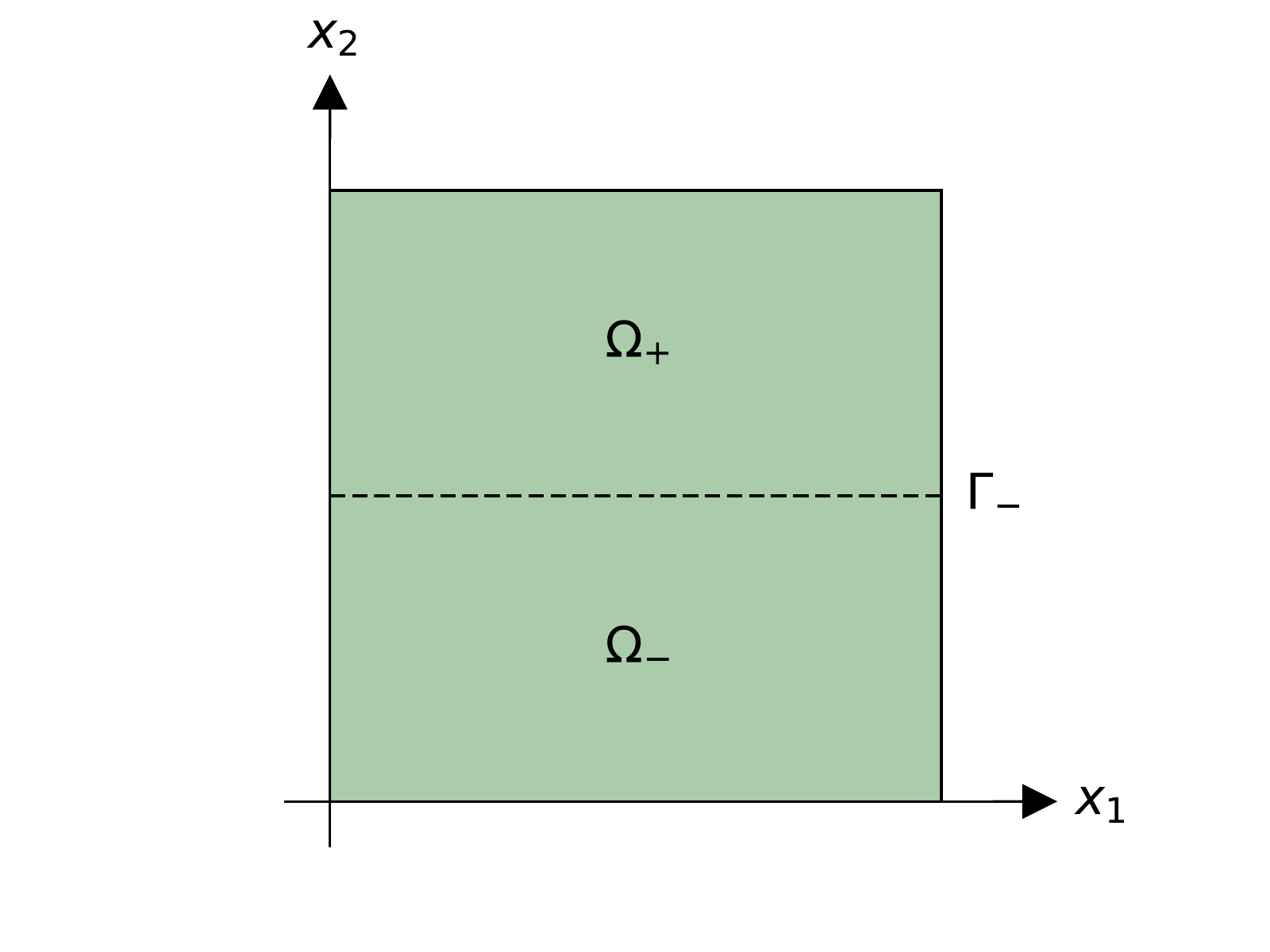}\\
        (b)
    \end{minipage}
    \caption{A locally periodic domain $\Omega^\ve$ (a), $\ve = 1/4$, and the corresponding homogeneous domain $\Omega$ (b),
                 with interface $\Gamma_-$ marked as a dashed line separating the regions $\Omega_+$ and $\Omega_-$.}
    \label{fig:domainz}
\end{figure}

\section{Periodic unfolding}\label{sec:unfolding}

The only apparent possible cause of oscillations in the solutions to \eqref{eq:originalproblem} and their flows 
is the periodicity in the domain, in the $x_1$ direction.
For the study of these oscillations we will use periodic unfolding.

The periodic unfolding at rate $\ve$ of a function $v : \mathbb{R}^2 \to \mathbb{R}$ along $x_1$ is
\begin{align*}
(T^\ve v)(x,y) & = v \big( \ve \big[\frac{x_1}{\ve}\big]+\ve y,x_2\big),
\end{align*}
where $[\cdot]$ denotes the integer part, and $v$ is extended by zero when necessary.
Using this change of variables for the periodic coefficient $\chi_{\Omega^\ve_+}$ in \eqref{eq:originalproblem}, 
the characteristic function of the region of $\Omega^\ve$ where coefficients are periodic,
\begin{align*}
\Omega^\varepsilon_+ & = \big\{ x \in \mathbb{R}^2 : 0 < x_1 < 1 ,\,\, \eta_-(x_1) < x_2 < \eta\big(x_1, \frac{x_1}{\ve}\big) \big\},
\end{align*}
gives $\chi_{\Omega^\ve_u}$, the characteristic function of the domain
\begin{align*}
\Omega^\ve_u & = \big\{ (x,y) \in \mathbb{R}^2 \times \mathbb{T} : 0 < x_1 < 1 ,\,\, \eta_-\big(\ve \big[\frac{x_1}{\ve}\big] + \ve y \big) < x_2 < \eta\big(\ve \big[\frac{x_1}{\ve}\big] + \ve y, y\big) \big\}.
\end{align*}
There holds
\begin{align}\label{eq:strong2scale}
T^\ve \chi_{\Omega^\ve_+} & = \chi_{\Omega^\ve_u} \to \chi_{\Omega_u} \quad \text{ strongly in } L^p(\mathbb{R}^2 \times \mathbb{T}), \quad 1 \le p < \infty, 
\end{align}
where
\begin{align*}
\Omega_u & = \{ (x,y) \in \mathbb{R}^2 \times \mathbb{T} : 0 < x_1 < 1 , \,\, \eta_- (x_1 ) < x_2 < \eta(x_1, y) \} \\
& = \{ (x,y) \in \mathbb{R}^2 \times \mathbb{T} : x \in \Omega_+ , \,\, y \in Y(x) \}.
\end{align*}

The property~\eqref{eq:strong2scale} is the strong unfolding convergence of the sequence and it will be used in
the passage from periodic domain to a fixed domain in integrals.
It expresses that $\chi_{\Omega_+^\ve}$ converges weakly in $L^p(\mathbb{R}^2)$ while not strongly, and that the oscillation spectrum of the sequence belongs to the integers if not empty.
To obtain~\eqref{eq:strong2scale} one uses the almost everywhere pointwise convergence of $\chi_{\Omega^\ve_u}$ to $\chi_{\Omega_u}$ and the Lebesgue dominated
convergence theorem, or views it as a consequence of Lemma~\ref{lm:unfolding} below.

The cost of replacing in integrals the $\ve$ depending unfolded domain $\Omega^\ve_u$ with the fixed domain $\Omega_u$ is described by the following lemma.

\begin{lemma}\label{lm:unfolding}
Let $\Omega$ contain $\Omega^\ve_+$.
Suppose that $\| v^\ve \|_{L^p( \Omega )} \le C$ and $p > 1$.
Then
\begin{align*}
\int_{\Omega^\ve_+} v^\ve \,dx & = \int_{\Omega_u} T^\ve v^\ve \,dxdy + O(\ve^{1 - 1/p}),
\end{align*}
as $\ve$ tends to zero.
\end{lemma}
\begin{proof}
Because $T^\ve \chi_{\Omega_+^\ve} = \chi_{\Omega_u^\ve}$, the discrepancy can be computed as follows:
\begin{align*}
\int_{\Omega_+^\ve} v^\ve \,dx - \int_{\Omega_u} T^\ve v^\ve \,dxdy
& = \int_{\Omega \times \mathbb{T}} T^\ve \chi_{\Omega_+^\ve} T^\ve v^\ve \,dxdy - \int_{\Omega_u} T^\ve v^\ve \,dxdy \\
& = \int_{\Omega \times \mathbb{T}} (\chi_{\Omega^\ve_u} - \chi_{\Omega_u}) T^\ve v^\ve \,dxdy.
\end{align*}
The differences $\Omega_u^\ve \setminus \Omega_u$ and $\Omega_u \setminus \Omega_u^\ve$ are contained in some strip of measure $O(\ve)$:
\begin{align*}
\{ (x,y) \in \mathbb{R}^2 \times \mathbb{T} : \mathrm{dist}( (x,y), \partial \Omega_u) < C\ve \},
\end{align*}
where $C$ may be chosen independent of $\ve$ by the Lipschitz continuity of $\eta$ and $\eta_-$,

By the H{\"o}lder inequality,
\begin{align*}
\Big| \int_{\Omega^\ve_+} v^\ve \,dx - \int_{\Omega_u} T^\ve v^\ve \,dxdy  \Big|
& \le \| \chi_{\Omega_u^\ve} - \chi_{\Omega_u} \|_{L^{ 1/(1 - 1/p) }(\Omega \times \mathbb{T})} \| v^\ve \|_{L^p( \Omega )}
= O(\ve^{(p-1)/p}),
\end{align*}
which gives the desired estimate.
\end{proof}

\section{Homogenization}\label{sec:homogenization}

In this section we establish the homogenization of problem \eqref{eq:originalproblem} to 
\eqref{eq:limitproblem} in the sense of weak convergence of the solutions and their flows.
The method we use is the unfolding Lemma~\ref{lm:unfolding} to pass to the fixed domain $\Omega_u$,
and the weak compactness in $L^2(\Omega_u)$ to characterize the asymptotic behavior of $u^\ve$.

Throughout this section (H1) is assumed to hold.

\begin{theorem}\label{tm:homogenization}
Suppose that (H1) holds.
Let $u^\ve \in H^1(\Omega^\ve, \Gamma)$ be the solutions to \eqref{eq:originalproblem},
and let $u^0 \in W(\Omega, \Gamma)$ be the solution to \eqref{eq:limitproblem}.
Then
\begin{align*}
&\emph{(i)} \quad \widetilde{u^\ve} \rightharpoonup h u^0 \quad \text{ weakly in } L^2(\Omega),\\
&\emph{(ii)} \quad \widetilde{\nabla u^\ve} \rightharpoonup h A^0 \nabla u^0 \quad \text{ weakly in } L^2(\Omega),
\end{align*}
as $\ve$ tends to zero, where $\sim$ denotes extension by zero.
\end{theorem}

The convergence of the flows in Theorem~\ref{tm:homogenization}(ii) means that $u^\ve$ converges weakly
to $u^0$ in $H^1(\Omega_-, \Gamma)$, and strongly in $L^2(\Omega_-)$ by the Relich theorem.

\begin{lemma}\label{lm:uveestimate}
For any $\ve = 1/k$, $k = 1, 2, \ldots$, 
there exists a unique solution $u^\ve \in H^1(\Omega^\ve, \Gamma)$ to \eqref{eq:originalproblem}.
For the solutions $u^\ve$, the following a priori estimate holds:
\begin{align*}
\| u^\ve \|_{H^1(\Omega^\ve,\,\Gamma)} \le C,
\end{align*}
where $C$ is independent of $\ve$.
\end{lemma}
\begin{proof}
The variational form of \eqref{eq:originalproblem} is: Find $u^\ve \in H^1(\Omega^\ve, \Gamma)$ such that
\begin{align}\label{eq:variationaloriginal}
\int_{\Omega^\ve} \nabla u \cdot \nabla \psi \,dx & = \int_{\Omega^\ve} f \psi \,dx,
\end{align}
for all $\psi \in H^1(\Omega^\ve, \Gamma)$.
Using the Poincar\'e inequality
\begin{align*}
\int_{\Omega^\ve} v^2 \,dx \le (\max \eta_+)^2 \int_{\Omega^\ve} \Big( \frac{\partial v}{\partial x_2} \Big)^2 \,dx, \quad v \in H^1(\Omega^\ve, \Gamma),
\end{align*}
one verifies that left hand side in \eqref{eq:variationaloriginal}
is an inner product on $H^1(\Omega^\ve, \Gamma)$.
The right hand side in \eqref{eq:variationaloriginal} is a bounded linear functional
on $H^1(\Omega^\ve, \Gamma)$.
The Riesz theorem guarantees the existence of a unique
solution $u^\ve \in H^1(\Omega^\ve, \Gamma)$. 
Using $u^\ve$ as a test function in \eqref{eq:variationaloriginal} gives
\begin{align*}
\| u^\ve \|_{H^1(\Omega^\ve, \,\Gamma)} & \le C \| f \|_{L^2( \Omega )},
\end{align*}
uniformly in $\ve$ by the uniform Poincar{\'e} constant, from which the desired a priori estimate is obtained.
\end{proof}

\begin{lemma}\label{lm:Tveuveestimate}
For the sequence of solutions $u^\ve$ to \eqref{eq:originalproblem}, the following a priori estimates hold
for the unfolded sequences:
\begin{align*}
\| T^\ve u^\ve \|_{ L^2(\Omega_u) } & \le C, \\
\| T^\ve \nabla u^\ve \|_{ L^2(\Omega_u) } & \le C,
\end{align*}
where $C$ is independent of $\ve$.
\end{lemma}
\begin{proof}
By the definition of unfolding,
\begin{align*}
\int_{\Omega^\ve_+} (u^\ve)^2 \, dx & =
\int_{\Omega_u^\ve} (T^\ve u^\ve)^2 \, dxdy
 = \int_{\Omega_u} (T^\ve u^\ve)^2 \, dxdy
+ \int_{\Omega_u^\ve \setminus \Omega_u} (T^\ve u^\ve)^2 \, dxdy,
\end{align*}
because $T^\ve u^\ve$ is zero outside $\Omega_u^\ve$. By using the estimate for the solutions $u^\ve$ in Lemma~\ref{lm:uveestimate}, the estimate $\| T^\ve u^\ve \|_{L^2(\Omega_u)} \le C$ is obtained.
The same computation with $|\nabla u^\ve|$ in place of $u^\ve$ gives the estimate $\| T^\ve \nabla u^\ve \|_{L^2(\Omega_u)} \le C$.
\end{proof}

Regarding the function space, $W(\Omega, \Gamma)$ is associated to the homogenized problem~\eqref{eq:limitproblem}.
In the cover $\Omega_+$ of the periodic region of $\Omega^\ve$,
the $H^1(\Omega)$ ellipticity of the operator to the homogenized problem \eqref{eq:limitproblem}
may be violated.
For 
\begin{align*}
h A^0 \xi \cdot \xi & = 
\begin{cases}
h \xi_2^2 \quad  \text{ in } \overline{\Omega_+}, \\
|\xi|^2 \quad \text{ in } \Omega_-,
\end{cases}
\end{align*}
and $h(x)$ tends to zero as $x$ approaches $M = \{ x : h(x) = 0 \}$, which might be nonempty.
The properties
\begin{align*}
& h(x) > 0 \quad \text{in } \Omega, \\
& h, h\inv \in L^1_{\mathrm{loc}}(\Omega),
\end{align*}
ensure that $W(\Omega, \Gamma)$ is a Hilbert space when equipped with the 
inner product $(u,v) = \int_\Omega A^0 \nabla u \cdot \nabla v \,dx$, and that $C^\infty_0(\Omega)$ is embedded into $W(\Omega, \Gamma)$. We denote by $C^\infty(\Omega, \Gamma)$, the $C^\infty(\Omega)$ functions vanishing in a neighborhood of $\Gamma$.

The existence of a solution $u^0 \in W(\Omega, \Gamma)$ to $\eqref{eq:limitproblem}$ is obtained by
weak compactness in the proof of Theorem~\ref{tm:homogenization} below, with uniqueness by linearity.
One might also obtain it in the direct way as follows.

\begin{lemma}\label{lm:limitlemma}
There exists a unique solution $u^0 \in W(\Omega, \Gamma)$ to $\eqref{eq:limitproblem}$.
\end{lemma}
\begin{proof}

The variational form of \eqref{eq:limitproblem} is: Find $u^0 \in W(\Omega, \Gamma)$ such that
\begin{align}\label{eq:variationallimit}
\int_{\Omega} hA^0 \nabla u^0 \cdot \nabla \psi \, dx & = \int_\Omega f \psi h \,dx,
\end{align}
for all $\psi \in W(\Omega, \Gamma)$.

The Hilbert space structure on $W(\Omega, \Gamma)$ has been chosen such that 
\begin{align*}
\int_{\Omega} hA^0 \nabla v \cdot \nabla v \, dx & = \| v \|_{W(\Omega, \, \Gamma)}^2.
\end{align*}
By the Hölder and Poincar\'e inequalities, the left hand side of \eqref{eq:variationallimit} defines an inner product on $W(\Omega, \Gamma)$.
The right hand side of \eqref{eq:variationallimit} is a bounded linear functional on $W(\Omega, \Gamma)$.
By the Riesz theorem, there exists a unique $u^0 \in W(\Omega, \Gamma)$ satisfying~\eqref{eq:variationallimit}.
\end{proof}

\begin{proof}[\bf Proof of Theorem~\ref{tm:homogenization}]
From Lemma~\ref{lm:uveestimate} and Lemma~\ref{lm:Tveuveestimate} we have the following a priori estimates for the solutions $u^\ve$
to problem \eqref{eq:originalproblem} and their unfoldings:
\begin{align*}
\| u^\ve \|_{H^1(\Omega^\ve, \, \Gamma)} & \le C,\\
\| T^\ve u^\ve \|_{ L^2(\Omega_u) } & \le C, \\
\| T^\ve \nabla u^\ve \|_{ L^2(\Omega_u) } & \le C.
\end{align*}
By weak compactness, there exist $u^0_- \in H^1(\Omega_-, \Gamma)$,
$u^0_+ \in L^2(\Omega_u)$, $p \in L^2(\Omega_u)$, and a subsequence of $\ve$ which we still denoted by $\ve$, such that
\begin{align}
u^\ve & \rightharpoonup u^0_- \quad  \qquad\quad \text{ weakly in } H^1(\Omega_-, \, \Gamma), \label{eq:weaks1}\\
T^\ve u^\ve & \rightharpoonup u^0_+ \quad \qquad\quad \text{ weakly in } L^2(\Omega_u),  \label{eq:weaks2}\\
T^\ve \nabla u^\ve & \rightharpoonup \big(p, \frac{\partial u^0_+}{\partial x_2} \big) \,\,\quad \text{ weakly in } L^2(\Omega_u), \label{eq:weaks3}
\end{align}
where the equality of second component of the weak limit of $T^\ve \nabla u^\ve$ and $\frac{\partial u^0_+}{\partial x_2}$,
and that $u^0_+$ does not depend on $y$, follow from the boundedness of the sequences and 
\begin{align}\label{eq:dy}
\frac{\partial}{\partial x_2} T^\ve u^\ve & = T^\ve \frac{\partial u^\ve}{\partial x_2}, &
\frac{\partial}{\partial y} T^\ve u^\ve & = \ve T^\ve \frac{\partial u^\ve}{\partial x_1}.
\end{align}
By (H1) and \eqref{eq:dy}, $u^0_+$ does not depend on $y$.\\

\noindent {\bf Claim 1: The average of $p$ in $y$ is zero: $\int_{Y(x)} p \, dy = 0$, a.e. $x \in \Omega_+$.}\\

Information about $p$ may be obtained by using oscillating test functions in the equation for $u^\ve$ (c.f.~\cite{marchenko1964boundary,tartar1977problemes}).
The prototype $\varphi^\ve = \ve \phi(x) \{ x_1/\ve \}$, with $\phi \in C^\infty_0(\Omega_+)$ and $\{ \cdot \}$ denoting fractional part,
would serve the purpose in view of \eqref{eq:W-M1bbb} because $T^\ve \varphi^\ve \to 0$ and 
$T^\ve \nabla \varphi^\ve \to (\phi,0)$ strongly in $L^2(\Omega_u)$.
As $\varphi^\ve$ are not necessarily continuous on $\Omega^\ve_+$, appropriate shifts are introduced as follows.
Note that 
\begin{align*}
\ve \big\{ \frac{x_1}{\ve} \big\} = x_1 - x_k^\ve, \quad \text{ if } x_1 \in [ x_k^\ve, x_{k+1}^\ve ) \text{ and } x^\ve_k = k \ve.
\end{align*}
Replace $x^\ve_k = k\ve$ with a grid where the graphs of $\eta(x_1,x_1/\ve)$ and $\eta_-(x_1)$ are close:
\begin{align}\label{eq:xk}
x^\ve_k & \in  \argmin_{x \in \ve [k,k+1]} \eta\big( x, \frac{x}{\ve} \big), \quad k = 0, \ldots, \frac{1}{\ve} - 1.
\end{align}
Let $\phi \in C^\infty_0(\Omega_+)$. Then with $x_k^\ve$ as in \eqref{eq:xk},
\begin{align}\label{eq:testfn}
\varphi^\ve(x) & = (x_1 - x_k^\ve)\phi(x), \quad \text{ if } x_1 \in [x_k^\ve, x^\ve_{k+1}),
\end{align}
belongs to $C^\infty(\Omega_+^\ve)$ for all small enough $\ve$.
Indeed, 
\begin{align*}
& \eta\big( x_k^\ve, \frac{x_k^\ve}{\ve} \big) - \eta_-(x_k^\ve) \\
& \quad = \min_{x \in \ve[k,k+1]} \eta\big( x, \frac{x}{\ve} \big) - \min_y \eta(x_k^\ve , y) \\
& \quad \le \max_{\xi \in \ve[k,k+1]}\big( \! \min_{y} \eta( \xi, y ) - \min_y \eta(x_k^\ve , y) \big) \\
& \quad \le C\ve,
\end{align*}
by the Lipschitz continuity of $\eta_-$, and $\phi$ is compactly supported.

With $\varphi^\ve$ given by \eqref{eq:xk}, \eqref{eq:testfn} as test functions in the equation \eqref{eq:W-M1bbb} for $u^\ve$,
and using that
\begin{align*}
T^\ve \varphi^\ve & \to 0, \\
T^\ve \nabla \varphi^\ve & \to (\phi, 0),
\end{align*}
strongly in $L^2(\Omega_u)$ as $\ve$ tends to zero because $|x_1 - x_k^\ve| \le \ve$, one obtains in the limit
\begin{align*}
\int_{\Omega_u} p \phi \, dxdy = \int_{\Omega_+} \, \int_{Y(x)} p\, dy \,  \phi \, dx = 0, \quad \phi \in C^\infty_0(\Omega_+). 
\end{align*}
which gives the first claim.\\[5mm]

\noindent \textbf{Claim 2: $\chi_{\Omega_-} u^0_- + \chi_{\Omega_+}u^0_+ \in W(\Omega, \Gamma)$ and is the solution to the homogenized problem~\eqref{eq:limitproblem}.}\\

To verify that $\chi_{\Omega_-} u^0_- + \chi_{\Omega_+}u^0_+ \in W(\Omega, \Gamma)$ it suffices to check that $u^0_- = u^0_+$ in $L^2(\Gamma_-)$.
Note that $u^0_+ \in L^2(\Omega_u)$ is continuously traced into $L^2(\Gamma_-)$ because $h > 0$ in a neighborhood of $\Gamma_-$ in $\Omega$ (see Remark~\ref{rem:trace}).

Because $u^\ve \rightharpoonup u^0_-$ weakly in $H^1(\Omega_-)$, the weak continuity of the trace gives
$u^\ve \rightharpoonup u^0_-$ weakly in $L^2(\Gamma_-)$,
and so $T^\ve u^\ve(x_1,\eta_-(x_1)) \rightharpoonup u^0_-(x_1,\eta_-(x_1))$
weakly in $L^2((0,1) \times \mathbb{T})$, as $\ve$ tends to zero.
Denote $\Gamma_\ve = \Omega_\ve \cap \Gamma_-$.
Let $\phi \in C^\infty_0(\Omega)$.
On the one hand,
\begin{align*}
 \int_{\Gamma_\ve} u^\ve \phi \nu_2 \, d\sigma 
& = \int_{0}^1 (\chi_{\Gamma_\ve} u^\ve  \phi \nu_2)(x_1,\eta_-(x_1)) \, dx_1 \\
& = \int_{0}^1\int_{\mathbb{T}} T^\ve (\chi_{\Gamma_\ve} u^\ve  \phi \nu_2)(x_1,\eta_-(x_1)) \, dy \, dx_1  \\
& \to \int_0^1 \int_{\mathbb{T}} (\chi_{Y(x)} u^0_- \phi \nu_2)(x_1,\eta_-(x_1)) \,dy \, dx_1 \\
& = \int_{\Gamma_-} h u^0_- \phi \nu_2 \, d\sigma,
\end{align*}
as $\ve$ tends to zero,
because $T^\ve \chi_{\Gamma_\ve}(x_1,\eta_-(x_1))$ converges to $\chi_{Y(x_1,\eta_-(x_1))}(y)$ strongly in $L^2((0,1)\times \mathbb{T})$.
On the other hand, 
\begin{align*}
\int_{\Gamma_\ve} u^\ve \phi \nu_2 \, d\sigma & = 
\int_{\Omega_u} T^\ve \frac{\partial u^\ve}{\partial x_2} T^\ve \phi \,dxdy
+
\int_{\Omega_u} T^\ve u^\ve T^\ve \frac{\partial \phi}{\partial x_2} \,dxdy + o(1) \\
& \to 
\int_{\Omega_u} \frac{\partial u^0_+}{\partial x_2}  \phi \,dxdy
+
\int_{\Omega_u} u^0_+ \frac{\partial \phi}{\partial x_2} \,dxdy  \\
& = \int_{\Gamma_-} h u^0_+ \phi \nu_2 \, d\sigma,
\end{align*}
as $\ve$ tends to zero.
Thus
\begin{align*}
\int_{\Gamma_-} (u^0_+ - u^0_-) \phi h \nu_2 \, d\sigma = 0, \quad \phi \in C^\infty_0(\Omega).
\end{align*}
It follows that $u^0_- = u^0_+$ in $L^2(\Gamma_-)$, for by the Lipschitz continuity of $\eta_-$,
\begin{align*}
h \nu_2 \ge  \frac{\min \{ h(x) : x \in \Gamma_- \} }{\sqrt{ 1 + \max (\eta_-')^2 }} > 0.\\
\end{align*}

A split of the variational form \eqref{eq:variationaloriginal} of problem \eqref{eq:originalproblem} into 
the periodic $\Omega^\ve_+$ and the fixed $\Omega_-$ reads
\begin{align*}
\int_{\Omega_-} \nabla u^\ve \cdot \nabla \psi \,dx +\int_{\Omega^\ve_+} \nabla u^\ve \cdot \nabla \psi \,dx
& =
\int_{\Omega^\ve} f \psi \,dx.
\end{align*}
After unfolding $\Omega^\ve_+$ and using Lemma~\ref{lm:unfolding} one arrives with $\psi \in C^\infty(\overline{\Omega},\Gamma)$ at 
\begin{align}\label{eq:W-M1bbb}
\int_{\Omega_-} \nabla u^\ve \cdot \nabla \psi \,dx
+ \int_{\Omega_u} T^\ve\nabla u^\ve \cdot T^\ve \nabla \psi \,dxdy
& = \int_{\Omega^\ve} f \psi \,dx + o(1).
\end{align}
By passing to the limit in \eqref{eq:W-M1bbb} 
with $\psi \in C^\infty(\overline{\Omega}, \Gamma)$, using the weak convergence of $u^\ve$, $T^\ve u^\ve$, $T^\ve \nabla u^\ve$ \eqref{eq:weaks1}--\eqref{eq:weaks3}, and that $p$ is of average zero in $y$, one obtains that $\chi_{\Omega_-} u^0_- + \chi_{\Omega_+}u^0_+$ satisfies 
\begin{align}\label{eq:weak-l-u}
\int_{\Omega_-}  \nabla u^0_- \cdot \nabla \psi \, dx + \int_{\Omega_u}    \frac{\partial u^0_+}{\partial x_2}  \frac{\partial \psi}{\partial x_2} \,  \,dy\, dx 
& = \int_{\Omega} \int_{Y(x)} f \psi  \,dy \, dx,
\end{align}
as $\ve$ tends to zero.
 
Let
\begin{align*}
\Omega_U & = \{ (x,y) : x \in \Omega, \, y \in Y(x) \},
\end{align*}
and
\begin{align*}
W(\Omega_U, \Gamma \times \mathbb{T}) = \{ v : \, & v \in L^2(\Omega_U), \, \frac{\partial v}{\partial x_1} \in L^2(\Omega_- \times \mathbb{T}), \, \frac{\partial v}{\partial x_2} \in L^2(\Omega_U), \\
& \nabla_y v = 0 \text{ in } \Omega_U,\, v = 0 \text{ on } \Gamma \times \mathbb{T} \}.
\end{align*}
By the density of $C^\infty(\overline{\Omega},\Gamma)$
in $W$ under hypothesis (H1),
\eqref{eq:weak-l-u} holds for any test function in $W(\Omega_U, \Gamma \times \mathbb{T})$:
\begin{align}\label{eq:limitW}
\int_{\Omega_U} \Big(  \chi_{\Omega_- \times \mathbb{T}}\frac{\partial u^0_-}{\partial x_1}\frac{\partial \psi}{\partial x_1} + \frac{\partial u^0_+}{\partial x_2}\frac{\partial \psi}{\partial x_2}  \Big) \,dxdy & = \int_{\Omega_U} f\psi \,dxdy,
\end{align}
for any $\psi \in W(\Omega_U, \Gamma \times \mathbb{T})$.
The equation~\eqref{eq:limitW} is well-posed in the Hilbert space $W(\Omega_U, \Gamma \times \mathbb{T})$.
Under (H1), the problem~\eqref{eq:limitW} is equivalent to \eqref{eq:limitproblem}, 
and $v$ belongs to the weighted space $W(\Omega, \Gamma)$ if and only if it belongs
to $W(\Omega_U, \Gamma \times \mathbb{T})$.
One concludes that $u^0 = \chi_{\Omega_-} u^0_- + \chi_{\Omega_+}u^0_+$.\\

\noindent \textbf{Claim 3: $\widetilde{u^\ve} \rightharpoonup h u^0$ weakly in $ L^2(\Omega)$.}\\

The uniqueness of the solution $u^0$ to the homogenized problem \eqref{eq:limitproblem}
ensures that the full $\ve$ sequences~\eqref{eq:weaks1}--\eqref{eq:weaks3} converge.\\

The weak limit of a sequence is obtained from the weak unfolding limit (weak two-scale limit) by taking the average over the cell of periodicity. Because $T^\ve u^\ve$ converges weakly to $u^0$, and $u^0$ does not depend on $y$,
\begin{align*}
\widetilde{u^\ve} \rightharpoonup \int_{Y(x)} u^0 \,dy = h u^0,
\end{align*}
weakly in $L^2(\Omega_+)$, as $\ve$ tends to zero.
It follows that $\widetilde{u^\ve} \rightharpoonup h u^0$ weakly in $ L^2(\Omega)$,
as $\ve$ tends to zero.\\

\noindent \textbf{Claim 4: $\widetilde{\nabla u^\ve} \rightharpoonup h A^0 \nabla u^0$ weakly in  $L^2(\Omega)$.}\\

By the same property of weakly converging unfolding, as was used in the previous paragraph, and that $p$ is of average zero in $y$,
\begin{align*}
\widetilde{\nabla u^\ve} \rightharpoonup \int_{Y(x)} \big( p, \frac{\partial u^0}{\partial x_2}\big) \,dy
= \big(0, h \frac{\partial u^0}{\partial x_2}\big),
\end{align*}
weakly in $L^2(\Omega_+)$, as $\ve$ tends to zero. It follows that
$\widetilde{\nabla u^\ve} \rightharpoonup h A^0 \nabla u^0$ weakly in  $L^2(\Omega)$, as $\ve$ tends to zero.
\end{proof}

\begin{remark}[About the $L^2$ trace]\label{rem:trace}
Let $\gamma(v)(x) = v(x)$ for $x \in \Gamma_-$, $v \in C^\infty(\overline{\Omega_+})$.
Then $\gamma : \{ v \in L^2(\Omega_+) : \frac{\partial v}{\partial x_2} \in L^2(\Omega_+) \} \to L^2(\Gamma_-)$
is linear and bounded, where the natural Hilbert space structures are employed.
Indeed, for $v \in C^\infty(\overline{\Omega_+})$,
\begin{align*}
v(x_1, \eta_-(x_1)) & =  - \int_{ \eta_-(x_1) }^{ x_2 } \frac{\partial v}{\partial x_2}(x_1,x_2) \,dx_2 + v(x_1, x_2).
\end{align*}
By the triangle inequality for the integral, the Bunyakovsky-Cauchy-Schwarz inequality, and the inequality of arithmetic and geometric means,
\begin{align*}
v(x_1, \eta_-(x_1))^2 \le C \Big( \int_{ \eta_-(x_1) }^{ \eta_+(x_1) } \big(\frac{\partial v}{\partial x_2}\big)^2(x_1,x_2) \,dx_2 + v^2(x_1, x_2)  \Big).
\end{align*}
An integration in $x_2$ over the interval $(\eta_-(x_1),\eta_+(x_1))$ gives
\begin{align*}
v(x_1, \eta_-(x_1))^2 \le C \Big( \int_{ \eta_-(x_1) }^{ \eta_+(x_1) } \big(\frac{\partial v}{\partial x_2}\big)^2(x_1,x_2) \,dx_2 + \int_{ \eta_-(x_1) }^{ \eta_+(x_1) } v^2(x_1, x_2) \,dx_2  \Big).
\end{align*}
An integration in $x_1$ over the interval $(0,1)$ gives
\begin{align*}
\int_{\Gamma_-} v^2 \,d\sigma \le C\Big( \int_{ \Omega_+} \big(\frac{\partial v}{\partial x_2}\big)^2 \,dx + \int_{\Omega_+} v^2 \,dx  \Big).
\end{align*}
The set $C^\infty(\overline{\Omega_+})$ is dense in the domain of $\gamma$.
\end{remark}

\section{Justification}\label{sec:justification}

In this section it is shown that the error in approximating the solutions $u^\ve$ to~\eqref{eq:originalproblem} and their flows
with the solution $u^0$ to the homogenized problem \eqref{eq:limitproblem} and its flow tends to zero,
measured in $L^2(\Omega^\ve, h)$.
The method we use is the convergence of energy.

Throughout this section (H1) is assumed to hold.

\begin{theorem}\label{tm:justification}
Suppose that (H1) holds.
Let $u^\ve \in H^1(\Omega^\ve, \Gamma)$ be the solutions to \eqref{eq:originalproblem},
and let $u^0 \in W(\Omega, \Gamma)$ be the solution to \eqref{eq:limitproblem}.
Then
\begin{align*}
\emph{(i)} & \quad \| u^\ve - u^0 \|_{ L^2(\Omega^\ve, \, h) } \to 0,\\
\emph{(ii)} & \quad \| \nabla u^\ve - A^0 \nabla u^0 \|_{ L^2(\Omega^\ve, \, h) } \to 0,
\end{align*}
as $\ve$ tends to zero.
\end{theorem}
\begin{proof}
By the weak convergence of $u^\ve$, $T^\ve u^\ve$, $T^\ve \nabla u^\ve$ \eqref{eq:weaks1}--\eqref{eq:weaks3},
the property that sum of lim inf is less than or equal to lim inf of sum,
 using that $u^\ve$, $u^0$ solve~\eqref{eq:W-M1bbb},~\eqref{eq:variationallimit},
\begin{align*}
& \int_{\Omega_u} p^2 \,dxdy
+ \int_{\Omega} h A^0 \nabla u^0 \cdot \nabla u^0 \, dx  \\
& \quad \le \liminf_{\ve \to 0} \int_{\Omega_u^\ve} |T^\ve \nabla u^\ve|^2 \,dxdy
+ \liminf_{\ve \to 0} \int_{\Omega_-} |\nabla u^\ve|^2 \,dx \\
& \quad \le \liminf_{\ve \to 0} \Big( \int_{\Omega_u^\ve} |T^\ve \nabla u^\ve|^2 \, dxdy
+ \int_{\Omega_-} | \nabla u^\ve |^2 \, dx \Big)\\
& \quad = \liminf_{\ve \to 0} \Big( \int_{\Omega_u^\ve} T^\ve f T^\ve u^\ve \,dxdy
+ \int_{\Omega_-} f u^\ve \,dx \Big) \\
& \quad = \int_{\Omega} f u^0 h \,dx \\
& \quad = \int_{\Omega} hA^0 \nabla u^0 \cdot \nabla u^0 \, dx.
\end{align*}
This shows that the energies converge.
It follows that each weak convergence in \eqref{eq:weaks1}--\eqref{eq:weaks3} is strong, and $p = 0$.
\end{proof}

As $T^\ve \frac{\partial u^\ve}{\partial x_1}$ tends to zero strongly in $L^2(\Omega_u)$, the proof of 
Theorem~\ref{tm:justification} shows that
\begin{align*}
\widetilde{\frac{\partial u^\ve}{\partial x_1}}  \to 0 \quad \text{ strongly in } L^2(\Omega_+),
\end{align*}
as $\ve$ tends to zero, which justifies the first term in the asymptotic expansion under hypothesis (H1).

\section{Average preserving extension}\label{sec:averagepreserving}

In this section we prove the homogenization Theorem~\ref{tm:homogenization} for an average preserving extension (see \cite{lipton1990darcy, DaPe-DCDS}).
The method we use is the strong unfolding convergence obtained in the proof of Theorem~\ref{tm:justification},
which is stated explicitly in the form of Lemma~\ref{lm:strongTve} below.

Throughout this section (H1) is assumed to hold.
By the end of the section, a remark about (H2) is included.

Let the local average of a function $v$ in the $x_1$ direction over $\Omega^\ve_+$ be denoted by
\begin{align}\label{eq:avgext1}
m_\ve(v)(x) = \frac{1}{h(x)} \int_{Y(x)}  (T^\ve v) (x,y) \,dy, \quad x \in \Omega_+,
\end{align}
where $m_\ve(v)(x)$ is set to zero at points where $h(x) = |Y(x)| = 0$.
The average preserving extension is defined by
\begin{align}\label{eq:avgext2}
\widetilde{u^\ve}^m (x)= \begin{cases}
\begin{aligned}
&u^\ve(x) &&\text{if}\;x \in \Omega^\ve,\\
& m_\ve(u^\ve)(x) &&\text{if}\; x \in\Omega \setminus \Omega^\ve.
\end{aligned}
\end{cases} 
\end{align}

For $v \in H^1(\Omega^\ve)$, $m_\ve(v)$ and $m_\ve(\nabla v)$ are well-defined,
using the Sobolev space property that $v$ has a representative that is absolutely continuous on almost all
line segments parallel to the coordinate axes and with square integrable partial derivatives.
In particular, $v$ belongs to $H^1(\{ (t, x_2) : t \in \mathbb{R} \} \cap \Omega^\ve_+)$ for almost all $x_2$ of relevance.

\begin{theorem}\label{tm:homogenization2}
Suppose that (H1) holds.
Let $u^\ve \in H^1(\Omega^\ve, \Gamma)$ be the solutions to~\eqref{eq:originalproblem},
and let $u^0 \in W(\Omega, \Gamma)$ be the solution to~\eqref{eq:limitproblem}.
Then
\begin{align*}
\emph{(i)} & \quad \widetilde{u^\ve}^m \to u^0 \quad \text{ strongly in } L^2(\Omega, h),\\
\emph{(ii)} & \quad \widetilde{\nabla u^\ve}^m  \to A^0 \nabla u^0 \quad  \text{ strongly in } L^2(\Omega, h ),
\end{align*}
as $\ve$ tends to zero, where $\sim{\!m}$ denotes the extension~\eqref{eq:avgext1}, \eqref{eq:avgext2}.
\end{theorem}

In the proof of Theorem~\ref{tm:justification}, the following strong convergence of the 
unfolded sequences were obtained, which in terms other than unfolding means that the oscillation spectrum of the sequences belong to the integers if not empty,
and that the sequences converge strongly two-scale.
\begin{lemma}\label{lm:strongTve}
Suppose that (H1) holds.
Let $u^\ve \in H^1(\Omega^\ve, \Gamma)$ be the solutions to~\eqref{eq:originalproblem},
and let $u^0 \in W(\Omega, \Gamma)$ be the solution to~\eqref{eq:limitproblem}.
Then
\begin{align*}
\emph{(i)} & \quad T^\ve u^\ve \to u^0 \quad \text{ strongly in } L^2(\Omega_u),\\
\emph{(ii)} & \quad T^\ve \nabla u^\ve \to \big( 0, \frac{\partial u^0}{\partial x_2}\big) \quad \text{ strongly in } L^2(\Omega_u),
\end{align*}
as $\ve$ tends to zero.
\end{lemma}

\begin{proof}[\bf Proof of Theorem~\ref{tm:homogenization2}]
Because the extension does not alter the functions $u^\ve$ in $\Omega_-$,
 the $\Omega_-$ parts of (i) and (ii) are included in Theorem~\ref{tm:justification}.
In $\Omega_+$, by definition,
\begin{align*}
\|\widetilde{u^\ve}^m - u^0\|_{L^2(\Omega_+, \, h)}^2 & =
\|u^\ve - u^0\|_{L^2(\Omega^\ve_+, \, h)}^2 +\|m_\ve(u^\ve) - u^0\|_{L^2(\Omega_+ \setminus \Omega_+^\ve, \, h)}^2.
\end{align*}
The first term on the right hand side tends to zero as $\ve$ tends to zero according to Theorem~\ref{tm:justification}.
The second term may be estimated as follows.
Because $u^0$ is independent of $y$, 
and $T^\ve u^\ve$ converges to $u^0$ strongly in $L^2(\Omega_u)$, 
by Theorem~\ref{tm:homogenization} and Lemma~\ref{lm:strongTve}(i),
the H{\"o}lder inequality gives
\begin{align*}
& \|m_\ve(u^\ve) -  u^0\|_{L^2(\Omega_+, \, h)}^2 \\[1mm]
& \quad = \int_{\Omega_+} \Big|\frac{1}{h(x)} \int_{Y(x)} T^\ve u^\ve \, dy - u^0 \Big|^2 h\, dx \\
& \quad = \int_{\Omega_+} \Big|\frac{1}{h(x)}\int_{Y(x)} T^\ve u^\ve \, dy -\frac{1}{h(x)} \int_{Y(x)} u^0 \, dy \Big|^2 h \, dx\\
& \quad \le  \int_{\Omega_u}  |T^\ve u^\ve - u^0 |^2 \,dxdy \\[1mm]
& \quad \to 0,
 \end{align*}
as $\ve$ tends to zero, which gives part~(i).

Part (ii) is obtained by repeating the above two steps for the components of $\nabla u^\ve$ in place of $u^\ve$,
using the second parts of Theorem~\ref{tm:justification} and Lemma~\ref{lm:strongTve}.
\end{proof}

\begin{remark}[Hypothesis (H2)]\label{rem:H2}
The theorems~\ref{tm:homogenization}, \ref{tm:justification}, \ref{tm:homogenization2}, and Lemma~\ref{lm:strongTve}, hold under the
slightly weaker hypothesis (H2).
The weak limit $u_+^0$ of $T^\ve u^\ve$ is the first component of the unique solution $(u_+^0,u_-^0)$ in 
the Sobolev space
\begin{align*}
W & = \big\{ (v,w) : v \in L^2(\Omega_u), \, \frac{\partial v}{\partial x_2} \in L^2(\Omega_u), \, \frac{\partial v}{\partial y} = 0, \\
& \qquad\qquad\quad\,  w \in H^1(\Omega_-,\Gamma), \, v = w \text{ on } \Gamma_- \times \mathbb{T}
 \big\}, 
\end{align*}
to the following problem: Find $(u_+,u_-) \in W$ such that
\begin{align}\label{eq:general}
\int_{\Omega_+} \int_{Y(x)} \frac{\partial u_+}{\partial x_2} \frac{\partial \varphi}{\partial x_2} \,dy \, dx
+ \int_{\Omega_-} \nabla u_- \cdot \nabla \phi \,dx & =
\int_{\Omega_+} f \int_{Y(x)} \varphi \,dy  \, dx + \int_{\Omega_-} f \phi \,dx,
\end{align}
for any $(\varphi, \phi) \in W$.
This statement relies on the density of $C^\infty(\overline{\Omega}, \Gamma)$ in $W$ (c.f. Section~\ref{sec:nonconnected} below).

By (H2), $(u_+^0(x,k/N + y), u_-^0)$, $k \in \mathbb{Z}$, and $(u_+^0(x,y_0-y), u_-^0)$ solve the problem~\eqref{eq:general}.
By uniqueness of solution, $u_+^0(x,y) = u_+^0(x,k/N + y)$, $k \in \mathbb{Z}$, and $u_+^0(x,y) = u_+^0(x,y_0-y)$.
Because all $Y_0(x)$ are connected, $\frac{\partial u^0_+}{\partial y} = 0$ implies that
$u^0_+$ is constant in $y$.
\end{remark}

\section{A numerical example}\label{sec:numerical}

To illustrate the rate of convergence in Theorem~\ref{tm:justification} under hypothesis (H2),
we consider the following example.
Let $u^\ve$ be the solutions to \eqref{eq:originalproblem} with $f = 1$ and 
\begin{align}\label{eq:etanumerical}
\eta(x,y) & = \frac{1}{4}(1+x \cos(4 \pi x))(3-\cos(2\pi y)) b(y),
\end{align}
where a bump is introduced at $y_0$ by
\begin{align*}
b(y) & = 
\begin{cases}
1 - 2 ( \delta - |y-y_0|) & \text{ if } |y-y_0| \le \delta,\\
1 & \text{ otherwise,}
\end{cases}
\end{align*}
with $\delta = 1/10$ and $y_0 = 1/2$, $y \in [0,1)$.
Then $\eta$ satisfies (H2) but not (H1).
Moreover, $h = 1$ on the graph $\Gamma_-$ of $\eta_-$, and $h = 0$ on the graph of $\eta_+$.
The domains $\Omega^\ve$ and $\Omega$ are illustrated in Figure~\ref{fig:domainnumerical}.

The solutions $u^\ve$, $u^0$ to the problems \eqref{eq:originalproblem}, \eqref{eq:limitproblem} are approximated by means of the finite element method using piecewise linear Lagrange elements.
The numerical approximations are denoted by $u^\ve_s$, $u^0_s$.
The numerically computed rates of convergence for the approximation in Theorem~\ref{tm:homogenization},
are illustrated in Figure~\ref{fig:error}, obtained using the numerical tool FreeFEM~\cite{hecht2012new}.
The data points were obtained using the values of $\ve$ and the number of degrees of freedom given in Table~\ref{tab:numerics}.

One observes that the rate of convergence for $u^\ve$ appears to be close to $\ve^{3/4}$,
and for $\nabla u^\ve$ close to $\ve^{1/4}$,
for the approximation measured in $L^2(\Omega^\ve, h)$ for the selected values of $\ve$.

\begin{figure}[!hb] 
    \centering
    \begin{minipage}{.5\textwidth}
        \centering
        \hspace*{5mm}
        \includegraphics[trim=30mm 5mm 10mm 0mm, clip,height=6.5cm]{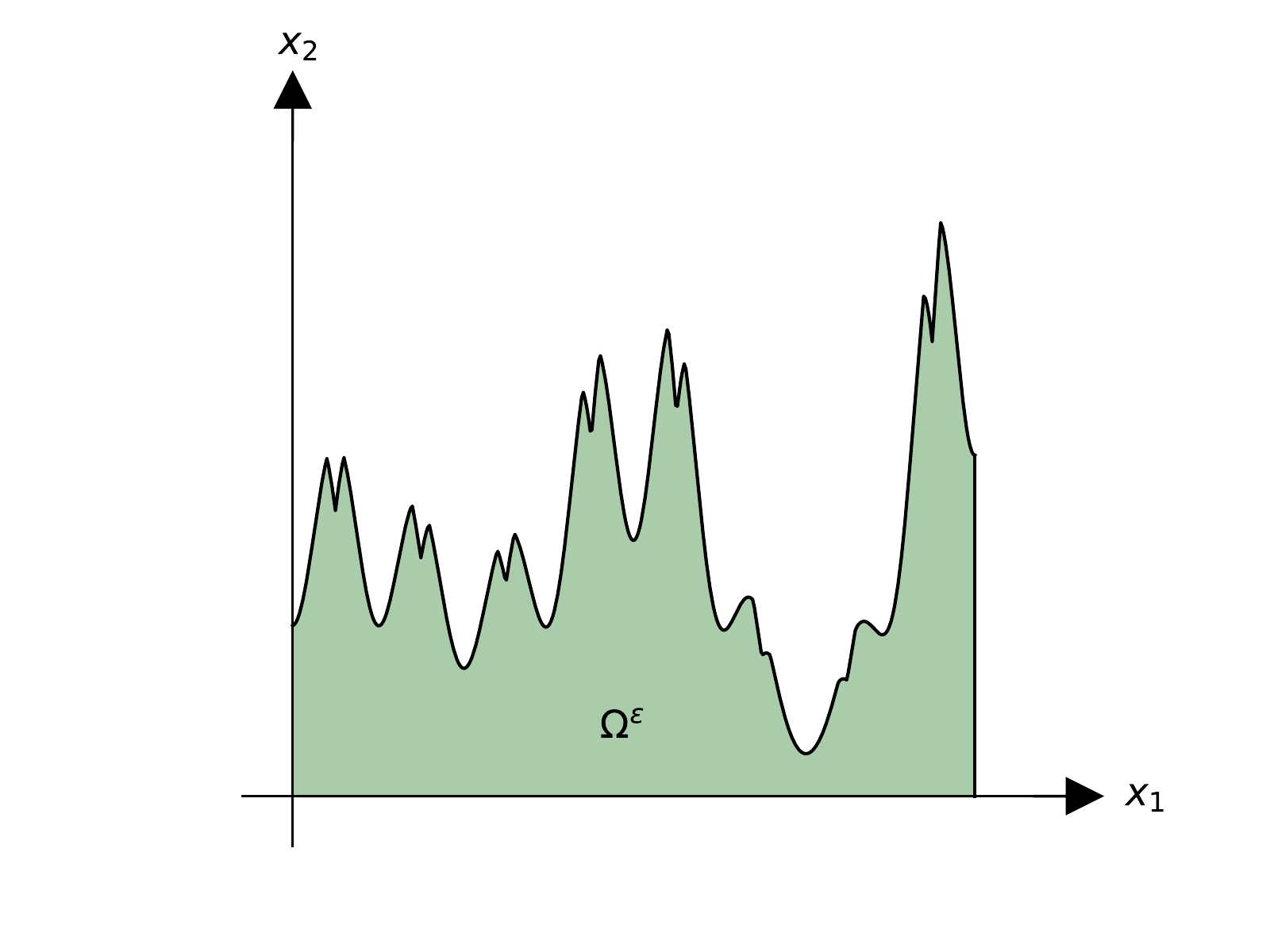}\\
        (a)
    \end{minipage}%
    \begin{minipage}{0.5\textwidth}
        \centering
        \hspace*{5mm}
        \includegraphics[trim=30mm 5mm 10mm 0mm, clip,height=6.5cm]{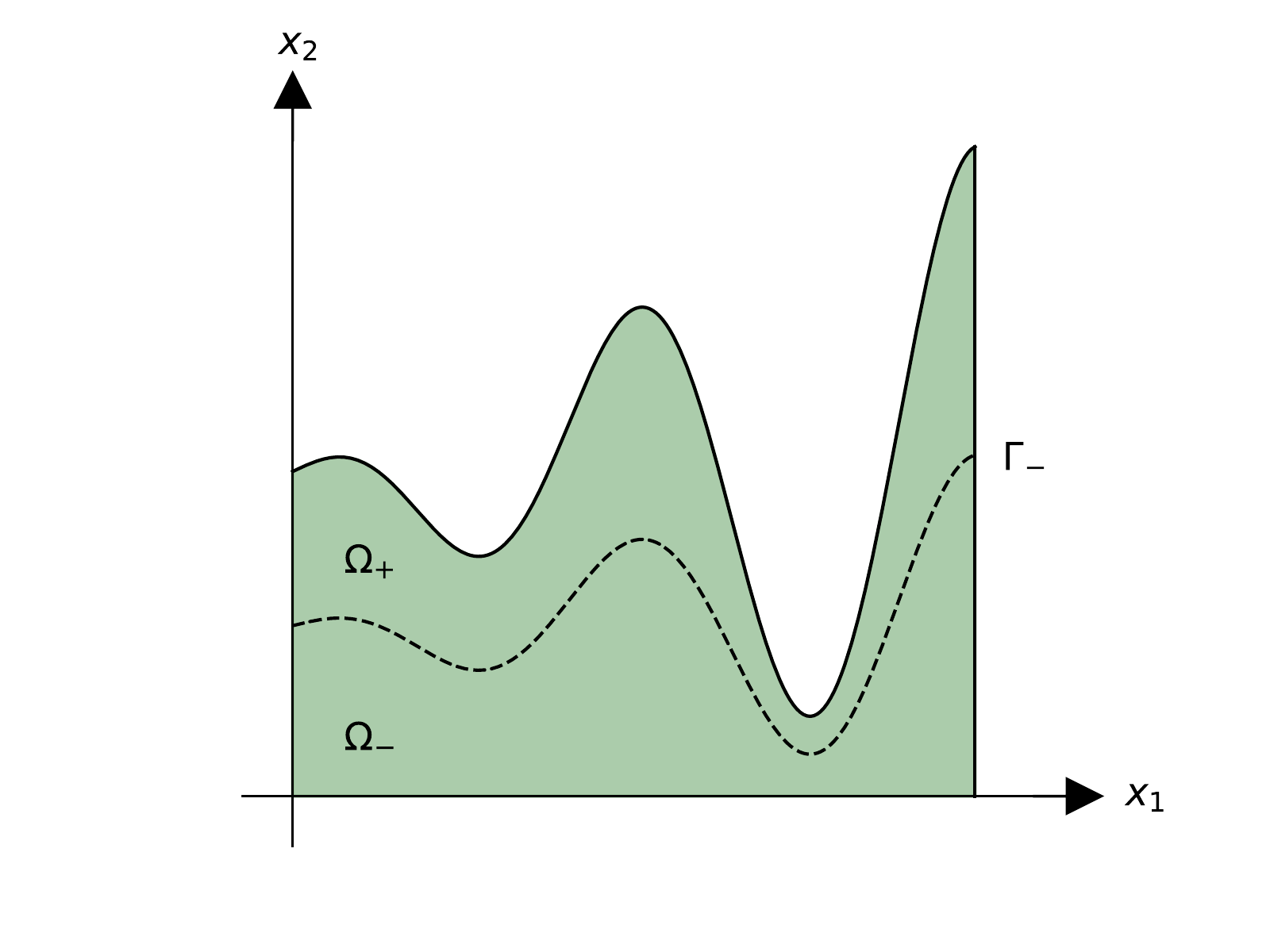}\\
        (b)
    \end{minipage}
    \caption{A locally periodic domain $\Omega^\ve$ (a) that satisfies (H2) but not (H1), $\ve = 1/8$, and the corresponding homogeneous domain $\Omega$ (b),
                 with $\Gamma_-$ marked as a dashed line separating the regions $\Omega_+$ and $\Omega_-$.
                 The particular function $\eta$ is specified in~\eqref{eq:etanumerical}.}
    \label{fig:domainnumerical}
\end{figure}

\begin{figure}[!hb] 
    \centering
        \includegraphics[trim=30mm 5mm 10mm 0mm, clip,height=6.5cm]{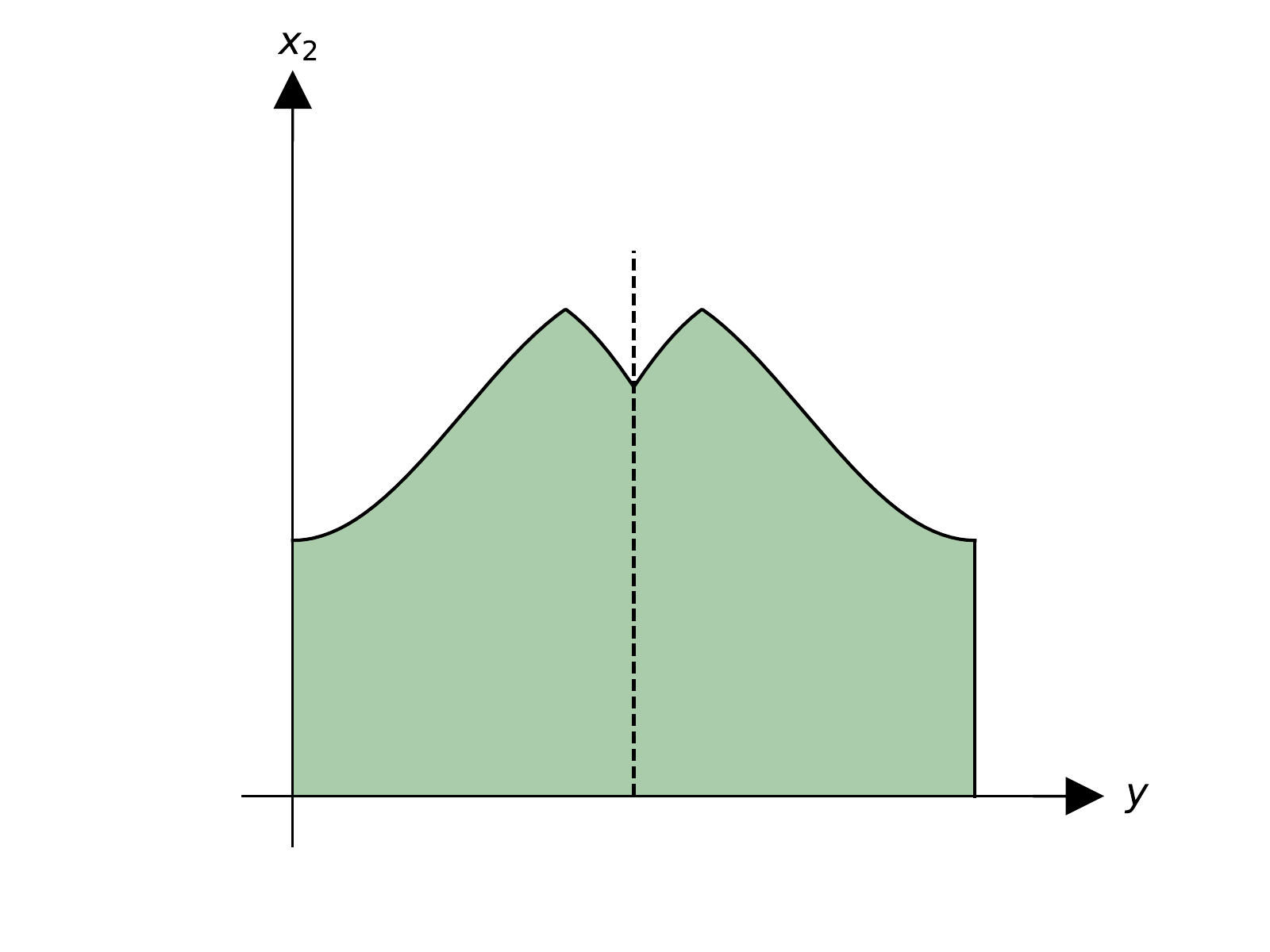}\\
    \caption{The region below the graph $x_2 = \eta(1/2,y)$ for $\eta$ given by~\eqref{eq:etanumerical} that satisfies (H2) but not (H1), with symmetry line $y = 1/2$ marked with a dashed line.}
    \label{fig:cellH2}
\end{figure}

\begin{figure}[!htb]
\centering
\includegraphics[trim=0mm 0mm 0mm 0mm,clip,width=.75\textwidth]{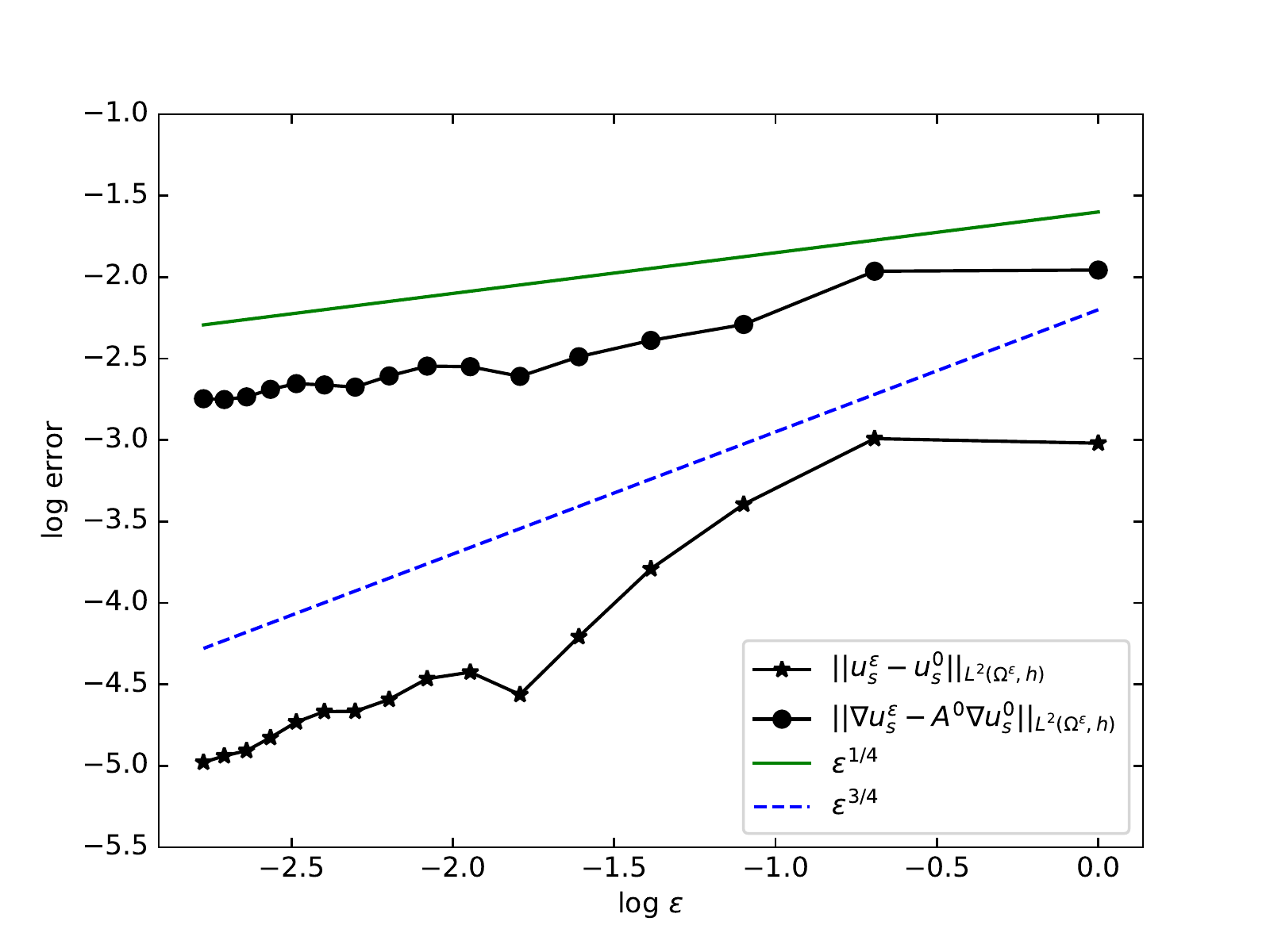}
\caption{The numerically computed rates of convergence for the approximations $u^0$ to $u^\ve$, and $A^0 \nabla u^0$ to $\nabla u^\ve$, in Theorem~\ref{tm:justification} for $f = 1$ and $\Omega^\ve$ given by~\eqref{eq:etanumerical} and illustrated for $\ve = 1/8$ in Figure~\ref{fig:domainnumerical}.}
\label{fig:error}
\end{figure}

\begin{table}[!ht]
\begin{tabular}{rccc}
$\ve$ & $\| u_s^\ve - u^0_s \|_{L^2(\Omega^{\varepsilon}, \, h)}$ & $\| \nabla u_s^\ve - A^0 \nabla u^0_s \|_{L^2(\Omega^{\varepsilon}, \, h)}$ & dof \\[3mm]
$0$ & -- & --                          & $78 \cdot 10^2$ \\
$1/1$   & $\,\,\,\,0.0488$           & $\,\,\,\,0.141$ & $39 \cdot 10^3$ \\
$1/2$   & $\,\,\,\,0.0502$           & $\,\,\,\,0.140$ & $46 \cdot 10^3$ \\
$1/3$   & $\,\,\,\,0.0336$           & $\,\,\,\,0.101$ & $33 \cdot 10^3$ \\
$1/4$   & $\,\,\,\,0.0226$           & $0.0918$ & $27 \cdot 10^3$ \\
$1/5$   & $\,\,\,\,0.0149$         & $0.0830$ & $59\cdot 10^3$ \\
$1/6$   & $\,\,\,\,0.0104$         & $0.0736$ & $69\cdot 10^3$ \\
$1/7$   & $\,\,\,\,0.0120$         & $0.0781$ & $77\cdot 10^3$ \\
$1/8$   & $\,\,\,\,0.0115$         & $0.0784$ & $11\cdot 10^4$ \\
$1/9$   & $\,\,\,\,0.0101$         & $0.0738$ & $13\cdot 10^4$ \\
$1/10$ & $\,0.00941$         & $0.0689$ & $15\cdot 10^4$ \\
$1/11$ & $0.00940$         & $0.0698$ & $20\cdot 10^4$ \\
$1/12$ & $0.00882$         & $0.0704$ & $23\cdot 10^4$ \\
$1/13$ & $0.00801$         & $0.0680$ & $24\cdot 10^4$ \\
$1/14$ & $0.00740$         & $0.0649$ & $35\cdot 10^4$ \\
$1/15$ & $0.00716$         & $0.0638$ & $35\cdot 10^4$ \\
$1/16$ & $0.00688$         & $0.0642$ & $37\cdot 10^4$ \\[5mm]
\end{tabular}
\caption{The numerically computed errors for the approximation in Theorem~\ref{tm:homogenization}, rounded to three significant digits.
The number of degrees of freedom, rounded to two significant digits, is abbreviated to dof.
The row $\ve = 0$ is for the numerical solution $u^0_s$ to the homogenized problem.}
\label{tab:numerics}
\end{table}

\clearpage

\section{A case of non-connected sections}\label{sec:nonconnected}

In this section we consider a case where the sections $Y(x)$ are allowed to be disconnected,
that is the hypotheses (H1) and (H2) are mildly relaxed.
First we describe the weak unfolding limit of the zero extended solutions to problem~\eqref{eq:originalproblem},
and then we describe the weak limit in terms of local domain densities using the ideas of Mel'nyk, splitting the oscillating part of the domain into branches
 (c.f.~\cite{de2005asymptotic,DurMel,mel2015asymptotic}).

Here, the full unfolded limit domain will be used, denoted by
\begin{align*}
\Omega_U = \{ (x,y) : x \in \Omega, \, y \in Y(x) \},
\end{align*}
and the effective matrix is given by
\begin{align}\label{eq:A0nonconnected}
A^0 & = 
\begin{pmatrix}
\chi_{\Omega_- \times \mathbb{T}} & 0 \\ 0 & 1
\end{pmatrix}.
\end{align}
The limit problem reads
\begin{align}\label{eq:limitproblemstrongdisconnected}
-\mop{div}_x ( A^0 \nabla_x u ) & = f \quad \hspace*{-0.5mm}\text{ in } \Omega_U, \notag\\
\nabla_y u & = 0 \quad \text{ in } \Omega_U, \\
A^0 \nabla_x u \cdot \nu_x & = 0 \quad \text{ on } \partial \Omega_U \setminus \Gamma \times \mathbb{T}, \notag\\
u & = 0 \quad \text{ on } \Gamma \times \mathbb{T},\notag
\end{align}
where $\nu_x$ denotes the projection of the outward unit normal to $\Omega_U$ onto $\mathbb{R}^2$.
The limit problem~\eqref{eq:limitproblemstrongdisconnected} is well-posed in the Sobolev space
\begin{align*}
& W(\Omega_U, \Gamma \times \mathbb{T}) \\
&\quad  = \{ v \in L^2(\Omega_U) : A^0\nabla_x v \in L^2(\Omega_U) ,\, \nabla_y v = 0 \text{ in } \Omega_U, \, v = 0 \text{ on } \Gamma \times \mathbb{T} \},
\end{align*}
equipped with the natural Hilbert space structure.

\begin{theorem}\label{tm:weaknonconnected}
Let $u^\ve \in H^1(\Omega^\ve, \Gamma)$ be the solutions to problem~\eqref{eq:originalproblem}.
Let $u^0 \in W(\Omega_U, \Gamma \times \mathbb{T})$ be the solution to~\eqref{eq:limitproblemstrongdisconnected}.
Then
\begin{align*}
&\emph{(i)} \quad \widetilde{u^\ve} \rightharpoonup \int_{Y(x)} u^0 \,dy \quad \text{ weakly in } L^2(\Omega), \\
&\emph{(ii)} \quad \widetilde{\nabla u^\ve} \rightharpoonup \int_{Y(x)} A^0 \nabla_x u^0 \,dy \quad \text{ weakly in } L^2(\Omega),
\end{align*}
as $\ve$ tends to zero.
\end{theorem}
\begin{proof}
The following a priori estimate holds
\begin{align*}
\| u^\ve \|_{H^1(\Omega_-)}
+ \| T^\ve u^\ve \|_{L^2(\Omega_u)}
+ \| T^\ve \nabla u^\ve \|_{L^2(\Omega_u)}
& \le C.
\end{align*}
There exist
\begin{align*}
u^0_- & \in H^1(\Omega_-, \Gamma), \\
u^0_+ & \in \big\{ v \in L^2(\Omega_u) : \frac{\partial v}{\partial x_2} \in L^2(\Omega_u), \, \nabla_y v = 0 \text{ in } \Omega_u \big\},\\
p & \in L^2(\Omega_u),
\end{align*}
such that, along a subsequence still denoted by $\ve$,
\begin{align*}
u^\ve & \rightharpoonup u^0_- \,\,\qquad\qquad \text{ weakly in } H^1(\Omega_-, \, \Gamma), \\
T^\ve u^\ve & \rightharpoonup u_+^0 \,\,\qquad\qquad \text{ weakly in } L^2(\Omega_u), \\
T^\ve \nabla u^\ve & \rightharpoonup \big( p, \frac{\partial u_+^0}{\partial x_2} \big) \qquad\!\! \text{ weakly in } L^2(\Omega_u).
\end{align*}
At this point $u_+^0$ and $p$ depend in general on the fast variable $y$.\\

The unfolded variational form of problem \eqref{eq:originalequation} is
by Lemma~\ref{lm:unfolding}, 
\begin{align}\label{eq:unfoldedform}
\int_{\Omega_-} \nabla u^\ve \cdot \nabla \psi \,dx
+ \int_{\Omega_u} T^\ve\nabla u^\ve \cdot T^\ve \nabla \psi \,dxdy
& = \int_{\Omega_U} T^\ve f \, T^\ve \psi \,dxdy + o(1),
\end{align}
as $\ve$ tends to zero, for any sufficiently smooth $\psi \in H^1(\Omega^\ve, \,\Gamma)$.\\

\noindent \textbf{Step 1: $p = 0$ a.e. in $\Omega_u$.}\\

We will show that $p = 0$ almost everywhere in $\Omega_u$.
Let $\phi \in C^\infty_0(\Omega_u)$ and consider the sequence of test
functions $\varphi^\ve(x) = \ve \phi\big(x, \frac{x_1}{\ve}\big)$.
Then 
\begin{align*}
T^\ve \varphi^\ve & \to 0, \\
T^\ve \nabla \varphi^\ve & \to \big( \frac{\partial \phi}{\partial y}, 0 \big),
\end{align*}
strongly in $L^2(\Omega_u)$, as $\ve$ tends to zero.
By passing to the limit in~\eqref{eq:unfoldedform} with test functions $\varphi^\ve$
one obtains
\begin{align*}
\int_{\Omega_u} p \frac{\partial \phi}{\partial y} \,dxdy & = 0, \quad \phi \in C^\infty_0(\Omega_u),
\end{align*}
as $\ve$ tends to zero. That is, $\nabla_y p = 0$ in $\Omega_u$.

Let $\phi \in C^\infty_0(\Omega_u)$ be such that $\nabla_y \phi = 0$.
Then consider the sequence of test functions $\varphi^\ve$ defined by
\begin{align*}
\varphi^\ve(x) & = (x_1 - x_k^\ve)\phi\big(x,\frac{x_1}{\ve}\big), \quad \text{ if } x_1 \in [x_k^\ve, x_{k+1}^\ve),
\end{align*}
where $x^\ve_k$ is a tagging such as in~\eqref{eq:xk}.
Then 
\begin{align*}
T^\ve \varphi^\ve & \to 0, \\
T^\ve \nabla \varphi^\ve & \to ( \phi, 0 ),
\end{align*}
strongly in $L^2(\Omega_u)$, as $\ve$ tends to zero.
By passing to the limit in~\eqref{eq:unfoldedform} with test functions $\varphi^\ve$
one obtains
\begin{align}\label{eq:pphi}
\int_{\Omega_u} p \phi \,dxdy & = 0, \quad \phi \in C^\infty_0(\Omega_u, \nabla_y),
\end{align}
as $\ve$ tends to zero.
By the density of $C^\infty_0(\Omega_u, \nabla_y)$
in the Sobolev space $\{ v \in L^2(\Omega_u) : \nabla_y v = 0 \text{ in } \Omega_u \}$,
one concludes that~\eqref{eq:pphi} holds for $\phi = p$, that is
\begin{align*}
\int_{\Omega_u} p^2 \,dxdy & = 0.
\end{align*}

\noindent \textbf{Step 2: Transmission condition.}\\

To determine the transmission condition for $u^0_-$ and $u^0_+$
on the internal interface 
\begin{align*}
\Gamma_-^u & = \{ (x,y) : x \in \Gamma_-, \, y \in Y(x) \},
\end{align*}
we will trace $u^\ve$ from either side.

Let $\phi \in C^\infty_0(\Omega_U)$, and set $\phi^\ve(x, \frac{x_1}{\ve})$.

On the one hand, from below,
\begin{align*}
 \int_{\Gamma_\ve} u^\ve \phi^\ve \nu_2 \, d\sigma 
& = \int_{0}^1 (\chi_{\Gamma_\ve} u^\ve  \phi^\ve \nu_2)(x_1,\eta_-(x_1)) \, dx_1 \\
& = \int_{0}^1\int_{\mathbb{T}} T^\ve (\chi_{\Gamma_\ve} u^\ve  \phi \nu_2)(x_1,\eta_-(x_1)) \, dy \, dx_1  \\
& \to \int_0^1 \int_{\mathbb{T}} (\chi_{Y(x)} u^0_- \phi \nu_2)(x_1,\eta_-(x_1),y) \,dy \, dx_1 \\
& = \int_{\Gamma_-} \int_{Y(x)} u^0_- \phi \nu_2 \, dy \, d\sigma,
\end{align*}
as $\ve$ tends to zero, because $T^\ve \chi_{\Gamma_\ve}(x_1,\eta_-(x_1))$ converges to $\chi_{Y(x_1,\eta_-(x_1))}(y)$ strongly in $L^2((0,1)\times \mathbb{T})$.

On the other hand, from above,
\begin{align*}
\int_{\Gamma_\ve} u^\ve \phi^\ve \nu_2 \, d\sigma & = 
\int_{\Omega_u} T^\ve \frac{\partial u^\ve}{\partial x_2} T^\ve \phi^\ve \,dxdy
+
\int_{\Omega_u} T^\ve u^\ve T^\ve \frac{\partial \phi^\ve}{\partial x_2} \,dxdy + o(1) \\
& \to 
\int_{\Omega_u} \frac{\partial u^0_+}{\partial x_2}  \phi \,dxdy
+
\int_{\Omega_u} u^0_+ \frac{\partial \phi}{\partial x_2} \,dxdy  \\
& = \int_{\Gamma_-} \int_{Y(x)}  u^0_+ \phi \nu_2 \, dy \, d\sigma,
\end{align*}
as $\ve$ tends to zero.

Thus
\begin{align*}
\int_{\Gamma_-} \int_{Y(x)} (u^0_+ - u^0_-) \phi \nu_2 \, dy \, d\sigma & = 0,
\quad \phi \in C^\infty_0(\Omega_U).
\end{align*}
It follows that
\begin{align}\label{eq:transmission}
u^0_- & = u^0_+ \quad \text{ in } L^2(\Gamma_-^u),
\end{align}
for by the Lipschitz continuity of $\eta_-$,
\begin{align*}
\nu_2 \ge  \frac{1 }{\sqrt{ 1 + \max (\eta_-')^2 }} > 0.
\end{align*}

The transmission condition \eqref{eq:transmission} means here that the trace of $u^0_+$ on each connected component of the sections $Y(x)$ for $x \in \Gamma_-$
is equal to the trace of $u^0_-$. That is, zero jump condition into each branch from below.\\

\clearpage
%
%
%
%

\noindent \textbf{Step 3: Limit problem.}\\

Let $\varphi \in C^\infty(\overline{\Omega_U})$ be such that $\nabla_y \varphi = 0$,
and $\varphi$ vanishes on $\Gamma \times \mathbb{T}$.
Set $\psi^\ve(x) = \varphi(x, \frac{x_1}{\ve})$, which belongs to $H^1(\Omega^\ve, \Gamma)$.
Using $\psi^\ve$ as test functions in \eqref{eq:unfoldedform}, 
one obtains
\begin{align}\label{eq:limitOmegaU}
\int_{\Omega_U} A^0 \nabla_x u \cdot \nabla_x \varphi \,dxdy & = 
\int_{\Omega_U} f \varphi \,dxdy,
\end{align}
in the limit as $\ve$ tends to zero,
where $u = \chi_{\Omega_u}u_+^0 + \chi_{\Omega_- \times \mathbb{T}}u_-^0
\in W(\Omega_U, \Gamma \times \mathbb{T})$ by Step 1.
In the upper part $\Omega_u$ one uses $p = 0$ by Step 2.
By the density of the considered set of test functions in $W(\Omega_U, \Gamma \times \mathbb{T})$, as $\Omega_U$ is Lipschitz, one has that $u = u^0$ is the unique solution to problem~\eqref{eq:limitproblemstrongdisconnected}.
One concludes that the weak convergences (i) and (ii) have been established
for the full sequences.
\end{proof}

\vspace*{5mm}

Remark that Lemma~\ref{lm:strongTve} holds without hypotheses (H1), (H2),
by the same argument, and it takes the following form, which is the result that corresponds to~Theorem~\ref{tm:justification} in this case.
\begin{lemma}\label{lm:strongTvenonconnected}
Let $u^\ve \in H^1(\Omega^\ve, \Gamma)$ be the solutions to~\eqref{eq:originalproblem},
and let $u^0 \in W(\Omega_U, \Gamma \times \mathbb{T})$ be the solution to~\eqref{eq:limitproblemstrongdisconnected}.
Then
\begin{align*}
\emph{(i)} & \quad T^\ve u^\ve \to u^0 \quad \text{ strongly in } L^2(\Omega_U),\\
\emph{(ii)} & \quad T^\ve \nabla u^\ve \to \big( \chi_{\Omega_- \times \mathbb{T}} \frac{\partial u^0}{\partial x_1} , \frac{\partial u^0}{\partial x_2}\big) \quad \text{ strongly in } L^2(\Omega_U),
\end{align*}
as $\ve$ tends to zero.
\end{lemma}

\vspace*{5mm}

We will conclude this section by translating the 
limit problem~\eqref{eq:limitproblemstrongdisconnected}
into a system in the original coordinates $x \in \Omega$, ($\overline{\Omega}$ is the Hausdorff limit of $\overline{\Omega^\ve}$), in a particular case
of a finite number of 'bumps' on the boundary.


\begin{lemma}\label{lm:cover}
Let $\omega$ be an open set in $\mathbb{R}^N$.
Then there exists an open cover $K'$ of $\omega$ such that for all $v \in L^2(\omega)$
with $\nabla_{x'} v = 0$ almost everywhere in $\omega$, $(x,x') \in \omega$, one has that 
$v$ is constant in $x'$ almost everywhere in each $V \in K'$.
\end{lemma}
\begin{proof}
A subset $V$ of $\omega$ is said to have the property $(P)$ if
the intersection of $V$ and any $k$-plane parallel to the $x'$-coordinate $k$-planes is connected, $x' \in \mathbb{R}^k$.
For $(x,x') \in \omega$, let $V(x,x')$ 
be a neighborhood of $(x,x')$ in $\omega$ that is not contained in any distinct neighborhood of $(x,x')$ in $\omega$ with the property (P).
By the Zorn lemma, at least one such maximal element exists for the set inclusion partial order on the set of all neighborhoods of $(x,x')$ with the property (P), because $\omega$ is open and if $F$ is a totally ordered subset, $\cup_{V \in F} V$ is an upper bound for $F$.
Then $K' = \{ V(x,x') : (x,x') \in \omega \}$ is an open cover of $\omega$.

Let $v \in L^2(\omega)$ be such that $\nabla_{x'}v = 0$ almost everywhere in $\omega$.
Then $v$ has a representative $\tilde{v}$ that is absolutely continuous on almost all line segments
 parallel to the $x'$-coordinate axes and whose classical partial derivatives parallel to the $x'$-coordinate axes
belong to $L^2(\omega)$.
It follows from Fubini's theorem that $\tilde v$ assumes constant value on each connected component of  each
 $k$-plane parallel to the $x'$-coordinate $k$-planes when intersected with $\omega$.
Therefore $v$ is constant in $x'$ almost everywhere on each $V \in K'$.
\end{proof}


Now we restrict to domains such that there is a finite open cover of $\Omega_U$ as in Lemma~\ref{lm:cover}
with $\omega = \Omega_U$ and $x' = y$, effectively discarding domains with a countably infinite number of 'bumps' on the boundary.
For instance, excluding $\eta(x,y) = x_1^2 \sin( \frac{y}{x_1} )$, and $\eta(x,y) = \mop{dist}(y,K)$, where $K$ is a Cantor set.

We say that $K$ is a partition of an open set $\omega$ in $\mathbb{R}^N$ if $K$ consists of
disjoint nonempty open subsets of $\omega$
such that $\overline{\omega} = \cup_{V \in K} \overline{V}$.
Under the assumption that there is a finite open cover $K'$ of $\Omega_U$ as in Lemma~\ref{lm:cover},
we can construct a finite partition $K$ of $\Omega_U$ as follows.
Given $K' = \{ V_1, \ldots, V_k \}$, let 
$K = \{ V_1, V_2 \setminus \overline{V_1}, \ldots, V_k \setminus \cup_{j=1}^{k-1} \overline{V_j} \}$.

Let $K$ be a finite partition of $\Omega_U$ with the same property as the open cover in Lemma~\ref{lm:cover}.
Denote the measure of the connected components of $Y(x)$ by
\begin{align*}
h_V & = |Y(x) \cap \{ y : (x,y) \in V \}|,
\end{align*}
for $V \in K$, that is the density of $\Omega^\ve$ in $\Omega$.
Denote the common boundaries of the subdomains $U, V \in K$ of the partition by 
\begin{align*}
\Gamma_{U,V} & = \partial U \cap \partial V.
\end{align*}
Let the effective matrix be given by
\begin{align*}
A^0 & = 
\begin{pmatrix}
\chi_{\Omega_-} & 0 \\
0 & 1 
\end{pmatrix}.
\end{align*}
Let 
\begin{align*}
W_K = \big\{ \, \{ v_V \}_{V \in K} : \,\, & v_V \in L^2(\pi_{\mathbb{R}^2} V, h_V), \\
& A^0 \nabla v_V \in L^2( \pi_{\mathbb{R}^2} V, h_V), \\
& [v] = 0 \text{ on } \pi_{\mathbb{R}^2} \Gamma_{U,V},   \\
& v_V = 0 \text{ on }  \Gamma \cap \pi_{\mathbb{R}^2} \overline{V},  \\ 
& U,V \in K\, \big\},
\end{align*}
where $\pi_{\mathbb{R}^2}$ denotes the projection onto $\mathbb{R}^2$,
equipped with the natural Hilbert space structure.

The limit problem
\begin{align}\label{eq:strongnonconnectedproblem}
-\mop{div}( h_V A^0 \nabla u_V) & = h_V f \quad \text{ in } \pi_{\mathbb{R}^2} V, \notag\\ 
u_V & = 0 \qquad\,\, \text{ on } \Gamma \cap \pi_{\mathbb{R}^2}\overline{V}, \notag\\
h_V  A^0 \nabla u_V \cdot \nu & = 0 \qquad\,\, \text{ on } (\partial \Omega \setminus \Gamma) \cap \pi_{\mathbb{R}^2}\overline{V}, \\
[u] & = 0 \qquad\,\, \text{ on } \pi_{\mathbb{R}^2} \Gamma_{U,V}, \notag\\ 
[h A^0 \nabla u \cdot \nu] & = 0 \qquad\,\, \text{ on } \pi_{\mathbb{R}^2} \Gamma_{U,V},\notag
\end{align}
is then well-posed in $W_K$.

\begin{corollary}\label{tm:notconnected}
Let $u^\ve \in H^1(\Omega^\ve, \Gamma)$ be the solutions to~\eqref{eq:originalproblem},
and let  $u^0 \in W(\Omega_U, \Gamma \times \mathbb{T})$ be the solution to~\eqref{eq:limitproblemstrongdisconnected}.
Let $K$ be a finite partition of $\Omega_U$ with the same property as the open cover in Lemma~\ref{lm:cover}
 with $\omega = \Omega_U$ and $x' = y$.
Then
\begin{align*}
\emph{(i)} & \quad \widetilde{u^\ve} \rightharpoonup \sum_{V \in K} h_V \chi_{\pi_{\mathbb{R}^2} V} u^0\big|_{V}
\quad \text{ weakly in } L^2(\Omega),\\
\emph{(ii)} & \quad \widetilde{\nabla u^\ve} \rightharpoonup \sum_{V \in K} h_V \chi_{\pi_{\mathbb{R}^2}V} A^0  \nabla u^0\big|_{V}
 \quad \text{ weakly in } L^2(\Omega),
\end{align*}
as $\ve$ tends to zero.
Moreover,
$u^0 = u_V^0$ in $V \in K$, where $\{ u_V^0 \}_{V \in K}$ in $W_K$ is the solution to~\eqref{eq:strongnonconnectedproblem}.
\end{corollary}
\begin{proof}
It follows from the weak unfolding convergence $T^\ve u^\ve \rightharpoonup u^0$ in $L^2(\Omega_U)$, that
\begin{align*}
\widetilde{u^\ve} & \rightharpoonup \int_{Y(x)} u^0 \,dy
= \int_{Y(x)} \sum_{V \in K} \chi_V u^0 \,dy
= \sum_{V \in K}  h_V \chi_{\pi_{\mathbb{R}^2}V} u^0\big|_{V},
\end{align*}
as $\ve$ tends to zero.
The same computation gives the convergence of the flows.
By writing $u^0 = \sum_{V \in K} \chi_V u^0$, one verifies that the limit problem~\eqref{eq:strongnonconnectedproblem} is equivalent to the 
limit problem~\eqref{eq:limitproblemstrongdisconnected} with $u_V^0 = \chi_V u^0$ on $V$, for any suitable partition $K$.
\end{proof}


\bibliographystyle{plain}
\bibliography{LocalPer}

\end{document}